\documentclass{amsart}
\usepackage{amssymb, amsmath}
\usepackage[latin1]{inputenc}
\usepackage[final]{graphicx}

\begin{document}

\newcommand{\grad}{\mbox{\rm grad}\,}
\newtheorem{theorem}{Theorem}[section]
\newtheorem{lemma}[theorem]{Lemma}
\newtheorem{observation}[theorem]{Observation}
\newtheorem{remark}[theorem]{Remark}
\newtheorem{definition}[theorem]{Definition}
\newtheorem{corollary}[theorem]{Corollary}
\newtheorem{example}[theorem]{Example}
\newtheorem{ansatz}[theorem]{Ansatz}
\newtheorem{xcase}{Case}[subsection]
\newtheorem{ycase}{Case}[section]
\makeatletter
\renewcommand{\theequation}{%
\thesection.\alph{equation}} \@addtoreset{equation}{section}
\makeatother
\title{Affine surfaces which are K\"ahler, para-K\"ahler, or nilpotent K\"ahler}
\author[Calvi\~no-Louzao]{E. Calvi\~no-Louzao, E. Garc\'{i}a-R\'{i}o, P. Gilkey, I. Guti\'errez-Rodr\'iguez, \\	R. V\'{a}zquez-Lorenzo}
\address{ECL: Conseller\'\i a de Cultura, Educaci\'on e Ordenaci\'on Universitaria, Edificio Administrativo San Caetano, 15781 Santiago de Compostela, Spain}
\email{{estebcl@edu.xunta.es}}
\address{EGR-IGR: Faculty of Mathematics, 	University of Santiago de Compostela, 	15782 Santiago de Compostela, Spain}
\email{eduardo.garcia.rio@usc.es; ixcheldzohara.gutierrez@usc.es}
\address{PG: Mathematics Department, \; University of Oregon, \;\; 	Eugene \; OR 97403, \; USA}
\email{gilkey@uoregon.edu}
\address{RVL: Department of Mathematics, IES de Ribadeo Dionisio Gamallo, 27700 Ribadeo, Spain}
\email{{ravazlor@edu.xunta.es}}
\thanks{Supported by projects ED431F 2017/03, and MTM2016-75897-P (Spain).}

\begin{abstract} 
Motivated by the construction of Bach flat neutral signature Riemannian extensions, we study the space of parallel trace free tensors of type $(1,1)$ on an affine surface. It is shown that the existence of such a parallel tensor field is characterized by the recurrence of the symmetric part of the Ricci tensor.

\end{abstract}
\maketitle

\section{Introduction}

\subsection{Bach flat modified Riemannian extensions}
Let $\mathcal{M}=(M,\nabla)$ be an affine surface (see Section~\ref{S1.2} below).
Let $\pi:T^*M\rightarrow M$ be the canonical projection
from the cotangent bundle to $M$. Let $(x^1,x^2)$ be local coordinates on $M$.
Expand $\omega=y_idx^i\in T^*M$ to define canonical coordinates $(x^1,x^2,y_1,y_2)$
on $T^*M$ where, by an abuse of notation, we identify $x^i$ with $\pi^*x^i$.
Let $T=T^i{}_j\partial_{x^i}\otimes dx^j$ be a tensor of type $(1,1)$ and let
$\phi=\phi_{ij}dx^i\circ dx^j$ be a symmetric 2-tensor field
where we adopt the {\it Einstein convention} and sum over repeated indices. The \emph{modified Riemannian extension} is
the invariantly defined Walker metric of neutral signature $(2,2)$ on $T^*M$ given locally by:
$$
g_{\nabla,\phi,T}:=2dx^i\circ dy_i+\left\{y_ry_sT^r{}_iT^s{}_j-2y_r\Gamma_{ij}{}^r+\phi_{ij}\right\}dx^i\circ dx^j\,.
$$

Let $\mathcal{N}:=(N,g)$ be a pseudo-Riemannian manifold with Levi-Civita connection $\nabla^g$ 
and let $\rho_{\nabla^g}$ denote the Ricci tensor. Let $W$ be the Weyl conformal curvature tensor. Then the \emph{Bach tensor} is defined by
$$
\mathfrak{B}_{ij}:={\nabla^g}{}^k{\nabla^g}{}^\ell W_{kij\ell}+\textstyle\frac12\rho_{\nabla^g}{}^{k\ell}W_{kij\ell}\,.
$$
The Bach tensor, which was introduced in \cite{Bach} to study conformal relativity, is trace free and is conformally invariant in dimension four. Bach flat metrics are critical points of the curvature invariant given by the $L^2$-norm of the Weyl tensor. Clearly locally conformally flat metrics as well as Einstein metrics are Bach flat. Moreover, half-conformally flat (i.e., self-dual or anti-self-dual) metrics and conformally Einstein metrics are Bach flat. 
There are few known examples of strictly Bach-flat manifolds, meaning the ones which are neither half conformally flat nor conformally Einstein.
Modified Riemannian extensions provide a tool to construct new examples of strictly Bach flat metrics as follows:

\begin{theorem}{\rm\cite{CGGV}} \label{T1.1}
	Let $\mathcal{M}$ be a connected affine surface equipped with a parallel tensor field $T$.
	Then $(T^*M,g_{\nabla,\phi,T})$ is Bach flat
	if and only if $T$ is either a multiple of the identity or nilpotent.
\end{theorem}

The modified Riemannian extensions with $T=c\operatorname{Id}$ are self-dual \cite{CGGV09} and thus one is mainly interested in the nilpotent case. Moreover, for each parallel nilpotent tensor field $T$, there is an infinite family of Bach flat modified Riemannian extensions since the deformation tensor field $\phi$ does not play any role in Theorem \ref{T1.1}. 

Let $\operatorname{Hess}_{\nabla^g}(h)=\nabla^g dh$ be the \emph{Hessian tensor} of a pseudo-Riemannian manifold $(N,g)$. 
We say that $(N,g,h)$ is a \emph{gradient Ricci soliton} if $\operatorname{Hess}_{\nabla^g}(h)+\rho_{\nabla^g}=\lambda g$ for some
$\lambda\in\mathbb{R}$. Gradient Ricci solitons are self-similar solutions of the Ricci flow and may be viewed as a natural 
generalization of Einstein metrics.
Four-dimensional half conformally flat gradient Ricci solitons are locally conformally flat in the Riemannian case \cite{chen-wang}. 
While all known examples of Bach flat gradient Ricci solitons in the Riemannian setting are locally conformally flat \cite{GRS1, GRS2}, 
there are non-trivial examples in the neutral signature case \cite{BVGR, CGGV}.

An important feature of the Bach flat examples in Theorem \ref{T1.1} is that they 
support gradient Ricci solitons which do not have any Riemannian counterpart. Let $\rho_s$
be the symmetric Ricci tensor of an affine surface (see Equation~(\ref{E1.a})).
The deformation tensor field $\phi$ is now  essential in the construction of gradient Ricci solitons 
as follows:

\begin{theorem}{\rm\cite{CGGV}}\label{T1.2}
Let $\mathcal{M}$ be a connected affine surface, 
let $0\ne T$ be a non-trivial parallel nilpotent tensor field, and let $f\in C^\infty(M)$. 
Then $(T^*M,g_{\nabla,\phi,T},h=\pi^*f)$ is a Bach-flat gradient Ricci soliton if and only if
$df(\ker(T))=0$ and
$$\phi(TX,TY)=-\{\operatorname{Hess}(f)+2\rho_{s}\}(X,Y)\text{ for all }X,Y\,.
$$
\end{theorem}

Quasi-Einstein metrics, although a generalization of gradient Ricci solitons, are
of interest in their own right. Conformally Einstein metrics and warped product Einstein metrics are special cases of quasi-Einstein metrics (see, for example, the discussion in \cite{BVGRGVR}). Let $(N,g)$ be a pseudo-Riemannian manifold. We say that $(N,g,h)$ is  \emph{quasi-Einstein} if $\operatorname{Hess}_{\nabla^g}(h)+\rho_{\nabla^g}-\mu\, dh\otimes dh=\lambda g$ for some
$\lambda,\mu\in\mathbb{R}$. As well as in the gradient Ricci soliton case, the Bach flat examples in Theorem \ref{T1.1} are quasi-Einstein for appropriate deformation tensor field $\phi$, which is now  essential.

\begin{theorem}{\rm\cite{BVGRGVR}}\label{T1.3}
Let $\mathcal{M}$ be a connected affine surface, let $T$ be a non-trivial parallel
nilpotent tensor field, and let $f\in C^\infty(M)$.
Then $(T^*M,g_{\nabla,\phi,T},h=\pi^*f)$ is a Bach-flat quasi-Einstein metric if and only if
$df(\ker(T))=0$ and
$$
\phi(TX,TY)=-\{\operatorname{Hes}(f)+2\rho_s-\mu\, df\otimes df\}(X,Y)
\text{ for all }X,Y\,.
$$
\end{theorem}

Motivated by Theorems \ref{T1.1}, \ref{T1.2} and \ref{T1.3}, one is interested in the existence of affine surfaces admitting a nilpotent K\"ahler structure (i.e. a parallel
tensor of Type~(1,1) which is nilpotent) and their explicit description.

\subsection{Affine geometry}\label{S1.2}
Let $\mathcal{M}=(M,\nabla)$ be an affine surface. Here $M$ is a smooth connected
surface and $\nabla$ is a torsion free connection on the tangent bundle of $M$. We shall often suppose
$M$ is simply connected to avoid difficulties with holonomy when passing from local to global results.
The Ricci tensor $\rho$ is defined by setting $\rho(X,Y):=\operatorname{Tr}\{ Z\rightarrow R(Z,X)Y\}$. Since the Ricci tensor need not be
 symmetric in general, we introduce the symmetrization 
$\rho_s$ and skew-symmetrization $\rho_{sk}$ by setting:
\begin{equation}\label{E1.a}\begin{array}{l}
\rho_{s}(X,Y):=\textstyle\frac12\{\rho(X,Y)+\rho(Y,X)\}\\[0.05in]
\rho_{sk}(X,Y):=\textstyle\frac12\{\rho(X,Y)-\rho(Y,X)\}.
\end{array}\end{equation}
Let $(x^1,x^2)$ be a system of local coordinates on $M$. To simplify the notation, we let $\partial_{x^i}:=\frac\partial{\partial x^i}$.
 Expand $\nabla_{\partial_{x^i}}\partial_{x^j}=\Gamma_{ij}{}^k\partial_{x^k}$
to define the Christoffel symbols of $\nabla$; since $\nabla$ is torsion free, $\Gamma_{ij}{}^k=\Gamma_{ji}{}^k$. 
Let $T$ be a tensor of Type~(1,1). Expand $T=T^i{}_j\partial_{x^i}\otimes dx^j$.
The associated endomorphism is given by $T\{a^j\partial_{x^j}\}=a^jT^i{}_j\partial_{x^i}$. We say that $T$
is {\it parallel} if $\nabla T=0$. Let
$\mathcal{P}(\mathcal{M})$ be the set of parallel tensors of type~$(1,1)$ on $\mathcal{M}$:
$$
\mathcal{P}(\mathcal{M})=\{T^i{}_j:\partial_{x^k}T^i{}_j+\Gamma_{k\ell}{}^i T^\ell{}_j-\Gamma_{kj}{}^\ell T^i{}_\ell=0\,,\ \forall\ i,j,k\}\,.
$$
We will prove the following result in Section~\ref{S2}.
\begin{lemma}\label{L1.4} If $\mathcal{M}=(M,\nabla)$ is a connected affine surface, then
$\mathcal{P}(\mathcal{M})$ is a unital algebra with $\dim\{\mathcal{P}(\mathcal{M})\}\le4$.
Let $T\in\mathcal{P}(\mathcal{M})$. The eigenvalues of $T$ are constant on $M$. If
$T$ vanishes at any point of $M$, then $T$ vanishes identically.
\end{lemma}

Let $\operatorname{Tr}\{T\}:=T^i{}_i$ be the trace of the endomorphism. Let
$$
\mathcal{P}^0(\mathcal{M}):=\{T\in\mathcal{P}(\mathcal{M}):\operatorname{Tr}\{T\}=0\}
$$
be the space of trace free parallel tensors of Type~(1,1). If $T\in\mathcal{P}(\mathcal{M})$,
$\operatorname{Tr}\{T\}$ is constant and expressing 
$T=\frac12\operatorname{Tr}(T)\operatorname{id}+(T-\frac12\operatorname{Tr}(T)\operatorname{id})$ decomposes
$$
\mathcal{P}(\mathcal{M})=\operatorname{id}\cdot\mathbb{R}\oplus\mathcal{P}^0(\mathcal{M})\,.
$$
If $0\ne T\in\mathcal{P}^0(\mathcal{M})$, 
then the eigenvalues of $T$ are $\{\pm\lambda\}$ so $\operatorname{Tr}\{T^2\}=2\lambda^2$.
If $2\lambda^2<0$ (resp. $2\lambda^2>0$), we can rescale $T$ so $T^2=-\operatorname{id}$ (resp. $T^2=\operatorname{id}$) and $T$ defines a 
{\it K\"ahler} (resp. {\it para-K\"ahler}) structure on $M$; the almost complex 
(resp. almost para-complex) structure being integrable
as $M$ is a surface \cite{C04,NN57}. Finally, 
if $\lambda=0$, then $T$ is nilpotent and defines what we will call a {\it nilpotent K\"ahler structure}; such tensors appear in
the construction of Bach flat manifolds using the Riemannian extension by Theorem~\ref{T1.1}.
The symmetric Ricci tensor plays a crucial role.
We will establish the following result in Section~\ref{S3}.

\begin{theorem}\label{T1.5} Let $\mathcal{M}=(M,\nabla)$ be a simply connected affine surface.
\begin{enumerate}
\item If $\dim\{\mathcal{P}^0(\mathcal{M})\}=1$, then exactly one of the following possibilities holds:
\begin{enumerate}
\item $\mathcal{M}$ admits a K\"ahler structure and $\operatorname{Rank}\{\rho_s\}=2$.
\item $\mathcal{M}$ admits a para-K\"ahler structure and $\operatorname{Rank}\{\rho_s\}=2$.
\item $\mathcal{M}$ admits a nilpotent K\"ahler structure and $\operatorname{Rank}\{\rho_s\}=1$.
\end{enumerate}
\item $\dim\{\mathcal{P}^0(\mathcal{M})\}\ne2$.
\item $\dim\{\mathcal{P}^0(\mathcal{M})\}=3$ if and only if $\rho_s=0$. This implies $\mathcal{M}$ admits K\"ahler, para-K\"ahler, and
nilpotent K\"ahler structures.
\end{enumerate}
\end{theorem}

Generically, of course, $\dim\{\mathcal{P}^0(\mathcal{M})\}=0$. 
Furthermore, there exist examples with $\operatorname{Rank}\{\rho_s\}=1$
(resp. $\operatorname{Rank}\{\rho_s\}=2$) where $\dim\{\mathcal{P}^0(\mathcal{M})\}=0$ 
as we shall show in Remark~\ref{R5.2}
 (resp. Remark~\ref{R1.10}). 
What is somewhat surprising is that the existence of parallel $(1,1)$ tensor fields is
 completely characterized by the geometry of the symmetric part of the Ricci tensor $\rho_s$.

Recall that a tensor field $\mathcal{T}$ on an affine manifold $\mathcal{M}$ is said to be \emph{recurrent} 
if $\nabla\mathcal{T}=\omega\otimes\mathcal{T}$ for some recurrence $1$-form $\omega$.
Let $\mathcal{M}$ be an affine surface where the skew-symmetric Ricci tensor $\rho_{sk}\neq 0$. Then
$\rho_{sk}$ defines a volume element. Furthermore, $\rho_{sk}$ is recurrent, i.e.
$\nabla\rho_{sk}=\omega\otimes\rho_{sk}$. The symmetric Ricci tensor is not recurrent in general .
We will also prove the following result in Section~\ref{S3}.

\begin{theorem}\label{T1.5b-EGR} 
Let $\mathcal{M}=(M,\nabla)$ be a simply connected affine surface with $\rho_s\neq 0$.
\begin{enumerate}
\item $\mathcal{M}$ admits a K\"ahler structure if and only if $\det\{\rho_s\}>0$ and $\rho_s$ is recurrent.
\item $\mathcal{M}$ admits a para-K\"ahler structure if and only if $\det\{\rho_s\}<0$ and $\rho_s$ is recurrent.
\item $\mathcal{M}$ admits a nilpotent K\"ahler structure if and only if $\rho_s$ is of rank one and recurrent.
\end{enumerate}
\end{theorem}

\begin{remark}\rm
Affine surfaces with $\mathcal{P}^0(\mathcal{M})\neq 0$ have appeared in the literature in
several contexts. For instance, affine surfaces with parallel shape operator have been investigated in
\cite{Jelonek}, where it is shown that any such surface is either an equiaffine sphere or the shape
operator is nilpotent, thus corresponding to Assertion (1.c) in Theorem \ref{T1.5}.
		
Let $\mathcal{M}$ be an affine surface equipped with a parallel volume form $\Omega$. 
Since $d\Omega=0$ and $\nabla\Omega=0$, $\mathcal{M}$ is a Fedosov manifold \cite{GRS}
and there is a notion of symplectic sectional curvature (see \cite{Fox,GRS}). 
A symplectic surface $(M,\nabla,\Omega)$ has zero symplectic sectional curvature if and only if the 
$\Omega$-Ricci operator $\Omega(\operatorname{Ric}^\Omega(X),Y)=\rho(X,Y)$ is a nilpotent 
K\"ahler structure. Moreover the symplectic sectional curvature is positive definite 
(resp., negative definite) if and only if $\operatorname{Ric}^\Omega$ is a 
K\"ahler (resp., para-K\"ahler) structure \cite{Fox}.
		
Moreover, since a symplectic surface has constant symplectic sectional curvature if and only 
if the Ricci tensor is parallel \cite{Fox}, it must be locally symmetric and thus locally 
homogeneous \cite{DGP18}. Hence
the cases of non zero constant symplectic curvature correspond to affine structure defined by the Levi-Civita connections 
of the sphere, the hyperbolic plane and the Lorentzian hyperbolic plane.
The case of zero symplectic sectional curvature corresponds to the Type $\mathcal{A}$ homogeneous surfaces given in the notation of \cite{BGG18} by
		$$
		\begin{array}{ll}
		\mathcal{M}_2^{-\frac{1}{2}}: & 
		\Gamma_{11}{}^1=-1, 
		\Gamma_{11}{}^2=0, 
		\Gamma_{12}{}^1=-\frac{1}{2}, 
		\Gamma_{12}{}^2=0, 
		\Gamma_{22}{}^1=0,
		\Gamma_{22}{}^2=0.
		\\
		\noalign{\medskip}
		\mathcal{M}_5^0: &
		\Gamma_{11}{}^1=-1, 
		\Gamma_{11}{}^2=0, 
		\Gamma_{12}{}^1=c,
		\Gamma_{12}{}^2=0,
		\Gamma_{22}{}^1=-1, 
		\Gamma_{22}{}^2=2c.
		\end{array}
		$$
	\end{remark}

In the proof of Theorem~\ref{T1.5} and Theorem~\ref{T1.5b-EGR}, we will give a complete 
local description of the setting where $\dim\{\mathcal{P}^0(\mathcal{M})\}=1$; this
naturally decomposes into 3 cases where $\mathcal{M}$ admits a K\"ahler structure, a para-K\"ahler structure, 
or a nilpotent K\"ahler structure.
We will give a complete local description of these 3 settings in Section~\ref{S3}. When examining the case where 
$\dim\{\mathcal{P}(\mathcal{M})\}=4$,
we will give a complete local description of the setting where $\rho_s=0$.
Although these results will provide a general solution to the problem of finding parallel tensors of type~$(1,1)$ on an affine surface,
it is of interest to find homogeneous solutions; this will be done presently but does not follow directly from this result owing to the difficulty
of determining when such a structure is homogeneous.

\subsection{Homogeneous affine geometries} 
We say that $\mathcal{M}$ is
locally homogeneous if given any two points $P$ and $Q$ of $M$, there exists a local diffeomorphism $\Psi$
from a neighborhood of $P$ to a neighborhood of $Q$ so $\Psi^*\nabla=\nabla$. 
The following result was first proved by Opozda \cite{Op04} in the torsion free
setting and subsequently extended by Arias-Marco and Kowalski \cite{AMK08} to 
surfaces with torsion. It is fundamental in the subject.

\begin{theorem}
Let $\mathcal{M}=(M,\nabla)$ be a locally homogeneous affine surface which is not flat. Then at least one of the following
three possibilities holds which describe the local geometry:
\begin{itemize}
\item[($\mathcal{A}$)] There exists a coordinate atlas so the Christoffel symbols
$\Gamma_{ij}{}^k$ are constant.
\item[($\mathcal{B}$)] There exists a coordinate atlas so the Christoffel symbols have the form
$\Gamma_{ij}{}^k=(x^1)^{-1}C_{ij}{}^k$ for $C_{ij}{}^k$ constant and $x^1>0$.
\item[($\mathcal{C}$)] $\nabla$ is the Levi-Civita connection of a metric of constant Gauss
curvature.
\end{itemize}\end{theorem}

\begin{remark}\label{R1.8}\rm These classes are not exclusive. There are no surfaces which are both of Type~$\mathcal{A}$
and Type~$\mathcal{C}$. The only surfaces which are of both Type~$\mathcal{B}$ and Type~$\mathcal{C}$ are
the hyperbolic plane $ds^2=\{(dx^1)^2+(dx^2)^2\}/(x^1)^2$ and the Lorentzian analogue $ds^2=\{(dx^1)^2-(dx^2)^2\}/(x^1)^2$.
A Type~$\mathcal{B}$ affine surface is also of Type~$\mathcal{A}$ if and only if 
$\Gamma_{12}{}^1=\Gamma_{22}{}^1=\Gamma_{22}{}^2=0$ (see \cite{BGG18}).
\end{remark}

We shall classify the Type~$\mathcal{A}$ and Type~$\mathcal{B}$ surfaces with
$\mathcal{P}^0(\mathcal{M})\ne\{0\}$. To avoid difficulties with holonomy (i.e. with the fundamental group),
we will assume henceforth that $M=\mathbb{R}^2$ in the Type~$\mathcal{A}$ setting and that
$M=\mathbb{R}^+\times\mathbb{R}$ in the Type~$\mathcal{B}$ setting. We shall be interested in geometries
which are not flat. Since we are in the 2-dimensional setting, this
is equivalent to imposing the condition that $\rho\ne0$. 

\subsection{Type~$\mathcal{A}$ geometries} We say that two Type~$\mathcal{A}$ structures on $\mathbb{R}^2$
are {\it linearly equivalent} if there exists an element $\Theta\in\operatorname{GL}(2,\mathbb{R})$
which intertwines the two structures. We will prove the following result in Section~\ref{S4}.

\begin{theorem}\label{T1.9}
Let $\mathcal{M}=(\mathbb{R}^2,\nabla)$ be a Type~$\mathcal{A}$ structure which is not flat.  Then
$\mathcal{P}^0(\mathcal{M})\ne\{0\}$ if and only if the Ricci tensor is of rank one. 
Furthermore, $\mathcal{M}$ is linearly equivalent to a structure where $\Gamma_{11}{}^2=0$ and $\Gamma_{12}{}^2=0$, and $\mathcal{P}^0(\mathcal{M})=T\cdot\mathbb{R}$, where $T=e^{-\Gamma_{11}{}^1x^1+(\Gamma_{22}{}^2-\Gamma_{12}{}^1)x^2}\partial_{x^1}\otimes dx^2$. 
\end{theorem}

\begin{remark}\rm\label{R1.10}
If $\mathcal{M}$ is a Type~$\mathcal{A}$ geometry which is not flat, then $\mathcal{M}$ is neither K\"ahler nor para-K\"ahler. Furthermore, any Type~$\mathcal{A}$ surface with $\operatorname{Rank}(\rho_s)=2$ satisfies 
$\dim\{\mathcal{P}^0(\mathcal{M})\}=0$. 
\end{remark}

\begin{remark}
\label{R1.11}
\rm
Let $\mathcal{M}$ be a Type~$\mathcal{A}$ surface with Ricci tensor of rank one and let
$T=e^{a_1x^1+a_2x^2}\partial_{x^1}\otimes dx^2$ be a nilpotent K\"ahler structure as in Theorem \ref{T1.9}. 
A straightforward calculation shows that the corresponding modified Riemannian extension 
$(T^*M,g_{\nabla,\phi,T})$ with deformation tensor field $\phi\equiv 0$ is anti-self-dual. 
This is due to the fact that any Type~$\mathcal{A}$ homogeneous geometry is projectively flat
(see Remark~\ref{EGR-3.5}).
Moreover it has been shown in \cite{BGG18} that any Type~$\mathcal{A}$ surface with Ricci tensor of rank one admits affine gradient Ricci solitons (i.e., smooth functions $f\in\mathcal{C}^\infty(M)$ satisfying $\operatorname{Hess}(f)+2\rho_s=0$) so that $df(\operatorname{ker}(\rho))=0$. Hence $(T^*M,g_{\nabla,0,T},h=\pi^*f)$ is an anti-self-dual gradient Ricci soliton which is never locally conformally flat. In this setting, the soliton is 	steady (i.e., $\lambda=0$) and isotropic (i.e., $\|d\pi^* f\|^2=0$).

In a more general setting, results in \cite{BVGRGVR} show that any  Type~$\mathcal{A}$ surface with Ricci tensor of rank one admits solutions of the affine quasi-Einstein equation (i.e., smooth functions $f\in\mathcal{C}^\infty(M)$ satisfying $\operatorname{Hess}(f)+2\rho_s-\mu\, df\otimes df=0$) so that $df(\operatorname{ker}(\rho))=0$. Hence $(T^*M,g_{\nabla,0,T},h=\pi^*f)$ is an anti-self-dual quasi-Einstein manifold which is never locally conformally flat.
\end{remark}

The situation is more complicated in the Type~$\mathcal{B}$ setting. For instance, there exist 
simply connected affine surfaces with $\operatorname{Rank}(\rho_s)=1$ but non-recurrent $\rho_s$
and $\dim\{\mathcal{P}^0(\mathcal{M})\}=0$. We will discuss these examples in Section~\ref{S5}. 
Also, in contrast with Type $\mathcal{A}$ surfaces, there are non flat Type $\mathcal{B}$ surfaces with $\rho_s=0$. 
This situation is discussed in Lemma \ref{L5.4}. A complete description of Type $\mathcal{B}$ surfaces with 
$\dim\{\mathcal{P}^0(\mathcal{M})\}=1$ is given in Section \ref{S5}, where explicit examples of 
K\"ahler, para-K\"ahler and nilpotent K\"ahler structures on Type $\mathcal{B}$ geometries are presented.

\section{Preliminary results}

\subsection{The proof of Lemma~\ref{L1.4}}\label{S2}
If $\mathbb{F}$ is a field, let $M_2(\mathbb{F})$ be the unital algebra of $2\times 2$ matrices with entries in $\mathbb{F}$
and let $M_2^0(\mathbb{F})\subset M_2(\mathbb{F})$ be the linear subspace of trace free matrices.
The sum and product of parallel tensors of Type~(1,1) is again parallel. Since $\operatorname{id}=(\delta^i{}_j)$ is parallel,
$\mathcal{P}(\mathcal{M})$ is a unital algebra. Fix a point $P\in M$. Since $M$ is connected, a parallel tensor is defined by
its value at a single point. Thus the map $T\rightarrow T(P)$ is a unital algebra homomorphism
which embeds $\mathcal{P}(\mathcal{M})$ into
$M_2(\mathbb{R})$ relative to the coordinate basis. Thus $\mathcal{P}(\mathcal{M})$ 
has dimension at most 4. Let $T\in\mathcal{P}(\mathcal{M})$.
Since $d\{\operatorname{Tr}(T)\}=\operatorname{Tr}(\nabla T)=0$, $\operatorname{Tr}(T)$ is constant.
By replacing $T$ by $T-\frac12\operatorname{Tr}(T)\operatorname{id}$, we may assume that $T\in\mathcal{P}^0(\mathcal{M})$
is trace free. The eigenvalues of $T$ are then $\{\lambda(P),-\lambda(P)\}$ so $\operatorname{Tr}\{T^2\}=2\lambda^2(P)$.
Since $T^2$ is parallel, this implies $\lambda^2(\cdot)$ is constant and hence the eigenvalues themselves are constant.
\qed
\subsection{Canonical local coordinates} Let $T$ be a tensor of Type~(1,1) on a smooth surface $M$ such that
the eigenvalues of $T$ are constant; this is equivalent, of course, to assuming either that $\operatorname{Tr}\{T\}$ and 
$\operatorname{Tr}\{T^2\}$ are constant on $M$ or that $\operatorname{Tr}\{T\}$ and $\det\{T\}$ are constant on $M$.
By subtracting a suitable multiple of the identity from $T$, we can assume $T$ is trace free. We have the following useful observation.
\begin{lemma}\label{L2.1} Let $0\ne T$ be a trace free tensor of Type~(1,1) on a smooth manifold $M$ with $\det\{T\}\in\{0,\pm1\}$.
\begin{enumerate}
\item If $\det\{T\}=0$, we can choose local coordinates so $T=\partial_{x^1}\otimes dx^2$.
\item If $\det\{T\}=1$, we can choose local coordinates so $T=\partial_{x^2}\otimes dx^1-\partial_{x^1}\otimes dx^2$.
\item If $\det\{T\}=-1$, we can choose local coordinates so $T=\partial_{x^1}\otimes dx^1-\partial_{x^2}\otimes dx^2$.
\end{enumerate}\end{lemma}

\begin{proof}Let $0\ne T$ be nilpotent.
Let $Y_1$ be a nonzero vector field which is defined locally
so that $TY_1\ne0$. Then $Y_2:=TY_1$ spans $\ker(T)$. Choose local coordinates $(y^1,y^2)$ so
that $Y_2=\partial_{y^2}$. Then $T\partial_{y^1}$ is a nonzero multiple of $\partial_{y^2}$, i.e.
$T\partial_{y^1}=f\partial_{y^2}$. Let $X_1=\partial_{y^1}+g\partial_{y^2}$ and $X_2=f\partial_{y^2}$ where $g$ remains to be determined.
Then $TX_1=X_2$. We have
$[X_1,X_2]=(\partial_{y^1}f+g\partial_{y^2}f-f\partial_{y^2}g)\partial_{y^2}$. Solve the ODE
$$
\partial_{y^2}g(y^1,y^2)=f^{-1}\{\partial_{y^1}f+g\partial_{y^2}f\}\text{ with }g(y^1,0)=0\,.
$$
This ensures $[X_1,X_2]=0$. Since $\{X_1,X_2\}$ are linearly independent, we can choose local coordinates
$(x^1,x^2)$ so $\partial_{x^1}=X_1$ and $\partial_{x^2}=X_2$. We then have $T\partial_{x^1}=\partial_{x^2}$ and $T\partial_{x^2}=0$;
Assertion~(1) follows 
after interchanging the roles of $x^1$ and $x^2$.
 
 If $\det\{T\}=1$, then $T^2=-\operatorname{id}$ and $T$ defines an almost complex structure. 
 Since $M$ is a surface, the Nirenberg-Newlander Theorem~\cite{NN57}
shows that we can choose local coordinates so $T\partial_{x^1}=\partial_{x^2}$ and $T\partial_{x^2}=-\partial_{x^1}$. Assertion~(2) now follows.

Let $\det\{T\}=-1$. Then $T^2=\operatorname{id}$ and $T$ defines an almost
para-complex structure. 
Since we are in dimension $2$, the para-complex structure is integrable and we can choose local coordinates so
$T\partial_{x^1}=\partial_{x^1}$ and $T\partial_{x^2}=-\partial_{x^2}$ (see, for example, \cite{C04}).
Assertion~(3) follows.
\end{proof}

\section{The proofs of Theorem~\ref{T1.5} and Theorem~\ref{T1.5b-EGR}}\label{S3}
In Section~\ref{S3.1},
we give a local form for the Christoffel symbols (see Equation~(\ref{E3.a})) that holds if and only if $\rho_s=0$.
If Equation~(\ref{E3.a}) holds, we compute $\rho$, we show that $\dim\{\mathcal{P}(\mathcal{M})\}=4$, and we
give an explicit basis for $\mathcal{P}(\mathcal{M})$ in this setting.  
 In Section~\ref{S3.2},
we give a local form for the Christoffel symbols (see Equation~(\ref{E3.b})) that holds if and only if 
$\mathcal{M}$ is nilpotent K\"ahler, i.e. $\mathcal{P}(\mathcal{M})$
contains a non-trivial nilpotent element.
If Equation~(\ref{E3.b}) holds, we compute $\rho_s$ and exhibit a non-trivial nilpotent element of
$\mathcal{P}^0(\mathcal{M})$ quite explicitly. We show that if
Equation~(\ref{E3.b}) holds, and if $\dim\{\mathcal{P}^0(\mathcal{M})\}\ge2$, then additional relations on the Christoffel
symbols pertain (see Equation~(\ref{E3.c})). If both Equation~(\ref{E3.b}) and Equation~(\ref{E3.c}) hold,
then $\dim\{\mathcal{P}^0(\mathcal{M})\}=3$,
$\rho_s=0$, and we exhibit an explicit basis for $\mathcal{P}^0(\mathcal{M})$. This shows that $\dim\{\mathcal{P}^0(\mathcal{M})\}\ne2$
if $\mathcal{P}^0(\mathcal{M})$ contains a non-trivial nilpotent element. In Section~\ref{S3.3}, we perform a similar analysis for K\"ahler
structures and in Section~\ref{S3.4}, we treat para-K\"ahler structures. Theorem~\ref{T1.5} and Theorem~\ref{T1.5b-EGR}
will follow from this analysis. 

\subsection{Trivial symmetric Ricci tensor}\label{S3.1}

\begin{lemma}\label{EGR-L3.1}
Let $(M,\nabla)$ be an affine surface which is not flat.
\begin{enumerate}
\item $\rho_s=0$ if and only if there is a coordinate atlas with locally defined $\varphi$ so:
\begin{equation}\label{E3.a}
\Gamma_{11}{}^1=0,\ \Gamma_{11}{}^2=0,\ \Gamma_{12}{}^1=\partial_{x^1}\varphi,\ 
\Gamma_{12}{}^2=0,\ \Gamma_{22}{}^1=\partial_{x^2}\varphi,\ \Gamma_{22}{}^2=\partial_{x^1}\varphi\,.
\end{equation}
\item If Equation~(\ref{E3.a}) holds, then $\rho=-\partial_{x^1}\partial_{x^1}\varphi\, dx^1\wedge dx^2$, and
$$
\mathcal{P}^0(\mathcal{M})=\operatorname{Span}\left\{\left(\begin{array}{cc}0&1\\0&0\end{array}\right),\,\,
\left(\begin{array}{cc}1&2\varphi\\0&-1\end{array}\right),\,\,
\left(\begin{array}{cc}-\varphi&-\varphi^2\\1&\varphi\end{array}\right)\right\}\,.
$$
\end{enumerate}
\end{lemma}

\begin{proof} Suppose $\rho_s=0$. Fix a local basis $\{e_1,e_2\}$ for $T_PM$. Let $\sigma(t):=\exp_P(te_2)$.
	Extend $e_1$ along $\sigma$ to be parallel and let $\Psi(s,t):=\exp_{\sigma(t)}(se_1(t))$.
	This gives a system of local coordinates where
	$\nabla_{\partial_t}\partial_t|_{s=0}=0$, $\nabla_{\partial_t}\partial_s|_{s=0}=0$, $\nabla_{\partial_s}\partial_s=0$,
	i.e.
	$$\begin{array}{lll}
	\Gamma_{22}{}^1(0,x^2)=0,&\Gamma_{22}{}^2(0,x^2)=0,&\Gamma_{12}{}^1(0,x^2)=0,\\[0.05in]
	\Gamma_{12}{}^2(0,x^2)=0,&\Gamma_{11}{}^1(x^1,x^2)=0,&\Gamma_{11}{}^2(x^1,x^2)=0\,.
	\end{array}$$
	We have $0=\rho_{s,11}=-(\Gamma_{12}{}^2)^2-\partial_{x^1}\Gamma_{12}{}^2=0$. Since $\Gamma_{12}{}^2(0,x^2)=0$,
	this ODE implies $\Gamma_{12}{}^2=0$. Setting $\rho_{s,12}=0$ then yields
	$\partial_{x^1}\{\Gamma_{12}{}^1-\Gamma_{22}{}^2\}=0$. Since $\Gamma_{12}{}^1(0,x^2)=0$ and $\Gamma_{22}{}^2(0,x^2)=0$,
	we conclude $\Gamma_{12}{}^1=\Gamma_{22}{}^2$. Setting $\rho_{s,22}=0$ yields 
	$-\partial_{x^2}\Gamma_{22}{}^2+\partial_{x^1}\Gamma_{22}{}^1=0$. Consequently $\Gamma_{22}{}^2=\partial_{x^1}\varphi$ and
	$\Gamma_{22}{}^1=\partial_{x^2}\varphi$ for some smooth function $\varphi$. 
	This yields the relations of Equation~(\ref{E3.a}). Conversely, if Equation~(\ref{E3.a}) holds, then a direct computation shows
	that $\rho_s=0$ and that the 3 endomorphisms of Assertion~(2)
	are parallel. Since these endomorphisms are linearly independent and $\dim\{\mathcal{P}^0(\mathcal{M})\}\le3$, Assertion~(2) holds.
\end{proof}

\subsection{Nilpotent K\"ahler structures}\label{S3.2}

\begin{lemma}\label{L3.2} Let $(M,\nabla)$ be an affine surface which is not flat. 
\begin{enumerate}
\item If $\mathcal{M}$ admits a nilpotent K\"ahler structure, there is a coordinate atlas so
\begin{equation}\label{E3.b}
\Gamma_{11}{}^1=0,\ \Gamma_{11}{}^2=0,\ \Gamma_{12}{}^2=0,\ 
\Gamma_{22}{}^2=\Gamma_{12}{}^1\,.
\end{equation}
\item If Equation~(\ref{E3.b}) holds, then $\rho_s=(\partial_{x^1}\Gamma_{22}{}^1-\partial_{x^2}\Gamma_{12}{}^1)\,dx^2\otimes dx^2$ and\newline
$T=\partial_{x^1}\otimes dx^2\in\mathcal{P}^0(\mathcal{M})$.
\item If Equation~(\ref{E3.b}) holds and if $\dim\{\mathcal{P}^0(\mathcal{M})\}\ge2$, then
\begin{equation}\label{E3.c}
\Gamma_{12}{}^1=-\partial_{x^1}\psi\,\,\text{ and }\,\,\Gamma_{22}{}^1=-\partial_{x^2}\psi\,\,\text{ for some smooth function}\,\, \psi\,.
\end{equation}
\item If Equations~(\ref{E3.b}) and (\ref{E3.c}) hold, then $\rho=\partial_{x^1}\partial_{x^1}\psi\, dx^1\wedge dx^2$ and\smallbreak
$\mathcal{P}^0(\mathcal{M})=\operatorname{Span}\left\{
\left(\begin{array}{cc}0&1\\0&0\end{array}\right),
\left(\begin{array}{cc}\psi&-\psi^2\\1&-\psi\end{array}\right),
\left(\begin{array}{cc}1&-2\psi\\0&-1\end{array}\right)\right\}$.
\end{enumerate}
\end{lemma}

\begin{remark}\rm\label{R3.3}
If $T$ is a nilpotent K\"ahler structure, then results of \cite{CGGV} show that the gradient Ricci soliton
$(T^*M,g_{\nabla,\phi,T}, h=\pi^* f)$ given in Theorem \ref{T1.2} is strictly Bach flat.
\end{remark}

\begin{proof} Let $0\ne T\in\mathcal{P}^0(\mathcal{M})$ be nilpotent.
By Lemma~\ref{L2.1}, we may choose coordinates so $T=\partial_{x^1}\otimes dx^2$.
Setting $\nabla T=0$ yields the following relations from which Equation~(\ref{E3.b}) follow
(see also \cite{CGGV}):
\begin{eqnarray*}
\nabla_{\partial x^1} T=0:&\left(
\begin{array}{cc}-\Gamma_{11}{}^2&\Gamma_{11}{}^1-\Gamma_{12}{}^2\\0&\Gamma_{11}{}^2
\end{array}\right)=\left(\begin{array}{cc}0&0\\0&0\end{array}\right)\,,
\\
\nabla_{\partial x^2} T=0:&
\left(\begin{array}{cc}-\Gamma_{12}{}^2&\Gamma_{12}{}^1-\Gamma_{22}{}^2\\0&\Gamma_{12}{}^2\end{array}\right)
=\left(\begin{array}{cc}0&0\\0&0\end{array}\right)\,.
\end{eqnarray*}

Assume Equation~(\ref{E3.b}) holds. A direct computation establishes Assertion~(2).
To prove Assertion~(3), assume in addition that $\dim\{\mathcal{P}^0(\mathcal{M})\}\ge2$ and choose $S\in\mathcal{P}^0(\mathcal{M})$ so $S$ and $T$
are linearly independent. We must establish the relations of Equation~(\ref{E3.c}).

\subsection*{Case 1} Suppose that $S$ is nilpotent. Express 
$$
S=\left(\begin{array}{cc}
S^1{}_1&S^1{}_2\\
S^2{}_1&-S^1{}_1
\end{array}\right);\quad 
ST
=\left(\begin{array}{cc}0&S^1{}_1\\0&S^2{}_1\end{array}\right)\,.
$$
Since $ST\in\mathcal{P}(\mathcal{M})$, $\operatorname{Tr}\{ST\}=S^2{}_1$ is constant. Thus $S^2{}_1=c$ for $c\in\mathbb{R}$ and
$$
S=\left(\begin{array}{cc}S^1{}_1&S^1{}_2\\c&-S^1{}_1\end{array}\right)\,.
$$
If $c=0$, then $\det(S)=-(S^1{}_1)^2=0$ implies $S^1{}_1=0$ so $S=S^1{}_2\, T$. Since $S$ and $T$ are parallel, $d S^1{}_2=0$ so $S^1{}_2\in\mathbb{R}$
and $S$ and $T$ are linearly dependent contrary to our assumption. Thus $c\ne0$ and we may rescale $S$ to assume $c=1$. Setting $\det(S)=0$ yields $S^1{}_2=-(S^1{}_1)^2$ so
$$
S=\left(\begin{array}{cc}
S^1{}_1&-(S^1{}_1)^2\\1&-S^1{}_1\end{array}\right)\,.
$$
We compute the covariant derivative $\nabla S=S^i{}_{j;k} \partial_{x^i}\otimes dx^j\otimes\partial_{x^k}$, where the components $S^i{}_{j;k}=\partial_{x^k}S^i{}_j+\Gamma_{k\ell}{}^i S^\ell{}_j-\Gamma_{kj}{}^\ell S^i{}_\ell$ to get $0=S^2{}_{2;1}=-\Gamma_{12}{}^1-\partial_{x^1}S^1{}_1$ and
$0=S^2{}_{2;2}=-\Gamma_{22}{}^1-\partial_{x^2}S^1{}_1$. This yields the additional relations given in Equation~(\ref{E3.c}).

\subsection*{Case 2} Suppose that  $S$ is not nilpotent.
The map $S\rightarrow S(P)$ is an algebra morphism which embeds $\mathcal{P}(\mathcal{M})$ in 
$M_2(\mathbb{R})$. Consequently, if
$\dim\{\mathcal{P}^0(\mathcal{M})\}=3$, then $\dim\{\mathcal{P}(\mathcal{M})\}=4$ and
$\mathcal{P}(\mathcal{M})$ contains a linearly independent nilpotent element
$S\in\mathcal{P}(\mathcal{M})$ and the argument given in Case~1 pertains. We therefore assume $\dim\{\mathcal{P}(\mathcal{M})\}=3$
and that any nilpotent element of $\mathcal{P}(\mathcal{M})$ is a constant multiple of $T$. Express
$$
S=\left(\begin{array}{cc}
S^1{}_1&S^1{}_2\\
S^2{}_1&-S^1{}_1
\end{array}\right)\,\, \text{ and }\,\, T=\left(\begin{array}{cc}0&1\\0&0\end{array}\right)\,.
$$
We compute
$$
ST=\left(\begin{array}{cc}0&S^1{}_1\\0&S^2{}_1\end{array}\right)\,.
$$
As $ST$ is parallel, $\operatorname{Tr}\{ST\}=S^2{}_1$ is constant so $S^2{}_1=c$ for some constant $c$ and
\begin{eqnarray*}
&&(ST-\textstyle\frac c2\operatorname{id})=\left(\begin{array}{cc}-\frac c2&S^1{}_1\\0&\frac c2\end{array}\right),\ 
STS=\left(\begin{array}{cc}cS^1{}_1&-(S^1{}_1)^2\\c^2&-cS^1{}_1\end{array}\right)\,.
\end{eqnarray*}
Since $\dim\{\mathcal{P}^0(\mathcal{M})\}=2$, there must exist a non-trivial real dependence relation 
of the form $0=a_1T+a_2(ST-\frac c2\operatorname{id})+a_3STS$, i.e.
$$
\left(\begin{array}{cc}0&0\\0&0\end{array}\right)=\left(\begin{array}{cc}
-\frac12 a_2c+a_3cS^1{}_1&a_1+a_2S^1{}_1-a_3(S^1{}_1)^2\\
a_3c^2&\frac 12a_2c-a_3cS^1{}_1\end{array}\right)\,.
$$
If $c\ne0$, the relation $a_3c^2=0$ implies $a_3=0$. The relation $\frac 12a_2c-a_3cS^1{}_1=0$ then implies $a_2=0$. And
then finally the relation $a_1+a_2S^1{}_1-a_3(S^1{}_1)^2=0$ implies $a_1=0$. Thus $c=0$ so we have
$$
S=\left(\begin{array}{cc}S^1{}_1&S^1{}_2\\
0&-S^1{}_1\end{array}\right),\ T=\left(\begin{array}{cc}0&1\\0&0\end{array}\right)\,.
$$
Since the eigenvalues of $S$ are constant, $S^1{}_1$ is constant as well. If $S^1{}_1=0$, then $\nabla S=0$ implies $S^1{}_2\in\mathbb{R}$ and hence $S$ and $T$ are not linearly
independent.  Thus we may assume $S^1{}_1=1$. 
We set $S^1{}_2=-2\psi$. Setting $\nabla S=0$ then shows that $\Gamma_{12}{}^1=-\partial_{x^1}\psi$ and
$\Gamma_{22}{}^1=-\partial_{x^2}\psi$ which yields, as desired, Equation~(\ref{E3.c}).

Assertion~(4) follows by a direct computation.
\end{proof}

\begin{proof}[Proof of Theorem~\ref{T1.5b-EGR}~(3)]
Let $\mathcal{M}$ be an affine surface with $\rho_s\ne0$. Assume $\mathcal{M}$ admits
a nilpotent K\"ahler structure.
Take adapted coordinates as in Lemma~\ref{L3.2} so that the Christoffel symbols are given by the
relations in Equation~(\ref{E3.b}). Then $\rho_s$ is recurrent of rank one with recurrence 1-form given by
$$
\omega=\partial_{x^1}\log\rho_{s,22}\,dx^1-(2\Gamma_{12}{}^1-\partial_{x^2}\log\rho_{s,22})\,dx^2\,.
$$

Conversely, let $\mathcal{M}$ be a recurrent affine surface with $\operatorname{Rank}\{\rho_s\}=1$. Take local coordinates $(x^1,x^2)$ so that $\ker\{\rho_s\}=\operatorname{Span}\{\partial_{x^1}\}$ (see Theorem~4.1 in \cite{W64}). If $\rho_s=\rho_{s,22}dx^2\otimes dx^2$, a straightforward calculation shows that $\nabla\rho_s=\omega\otimes\rho_s$ for some $1$-form $\omega$ if and only if $\Gamma_{11}{}^2=0$ and $\Gamma_{12}{}^2=0$.
Furthermore, one has
$$
\begin{array}{l}
\rho_{s,12}=\frac{1}{2}\left(\partial_{x^1}(\Gamma_{12}{}^1-\Gamma_{22}{}^2)-\partial_{x^2}\Gamma_{11}{}^1 \right)\,,
\quad\quad\quad
\rho_{s,11}=0\,,
\\
\noalign{\medskip}
\rho_{s,22}=\Gamma_{11}{}^1\Gamma_{22}{}^1+\Gamma_{12}{}^1\left(\Gamma_{22}{}^2-\Gamma_{12}{}^1\right)+\partial_{x^1}\Gamma_{22}{}^1-\partial_{x^2}\Gamma_{12}{}^1\,.
\end{array}
$$
Since $\rho_{s,12}=0$ one has the additional relation
$$
\Gamma_{11}{}^1=\mu(x^1)+ \int\partial_{x^1}\left(\Gamma_{12}{}^1-\Gamma_{22}{}^2\right) dx^2\,.
$$
Change the coordinates as
$(u^1,u^2)=(x^1+a(x^1),x^2)$ so that
$$
\begin{array}{ll}
du^1=(1+a')dx^1 &du^2=dx^2\\
\partial_{u^1}=(1+a')^{-1}\partial_{x^1} &\partial_{u^2}=\partial_{x^2}
\end{array}
$$
Now, one has that 
$$
{}^u\Gamma_{11}{}^2=0,\quad {}^u\Gamma_{12}{}^2=0,\quad
{}^u\Gamma_{12}{}^1={}^x\Gamma_{12}{}^1,\quad
{}^u\Gamma_{22}{}^2={}^x\Gamma_{22}{}^2
$$
and
$$
\begin{array}{l}
{}^u\Gamma_{11}{}^1=\frac{1}{1+a'(x^1)}\left({}^x\Gamma_{11}{}^1-\frac{a''(x^1)}{1+a'(x^1)}\right)
\\
\noalign{\medskip}
\phantom{{}^u\Gamma_{11}{}^1}
=\displaystyle
\frac{1}{1+a'(x^1)}\!\left(\!\mu(x^1)-\frac{a''(x^1)}{1+a'(x^1)} +\!\int\partial_{x^1}\left({}^u\Gamma_{12}{}^1-{}^u\Gamma_{22}{}^2\right) dx^2\!\right)
\\
\noalign{\medskip}
\phantom{{}^u\Gamma_{11}{}^1}
=\displaystyle
\frac{1}{1+a'(x^1)}\!\left(\!\mu(x^1)-\frac{a''(x^1)}{1+a'(x^1)}\!\right)
+\frac{1}{1+a'(x^1)}\!\int\partial_{x^1}\left({}^u\Gamma_{12}{}^1-{}^u\Gamma_{22}{}^2\right) dx^2
\\
\noalign{\medskip}
\phantom{{}^u\Gamma_{11}{}^1}
=\displaystyle
\frac{1}{1+a'(x^1)}\!\left(\!\mu(x^1)-\frac{a''(x^1)}{1+a'(x^1)}\!\right)
+\int\!\partial_{u^1}\left({}^u\Gamma_{12}{}^1-{}^u\Gamma_{22}{}^2\right) du^2
\end{array}
$$
Hence choosing $a(x^1)$ to be a solution of $a''-\mu a'-\mu=0$ one may assume that 
$$
\Gamma_{11}{}^1= \int\partial_{x^1}\left(\Gamma_{12}{}^1-\Gamma_{22}{}^2\right) dx^2\,.
$$

Let $T=T^1{}_2\,\partial_{x^1}\otimes dx^2$ be a nilpotent tensor field on $\mathcal{M}$. Then $T$ is parallel if and only if 
$$
T^1{}_{2;2}=\displaystyle
\partial_{x^2}T^1{}_2+(\Gamma_{12}{}^1-\Gamma_{22}{}^2)T^1{}_2=0\,,
\quad
\mbox{and}
\quad
T^1{}_{2;1}=\displaystyle
\partial_{x^1}T^1{}_2+T^1{}_2\Gamma_{11}{}^1
=0\,.
$$
Use the equation $T^1{}_{2;2}=0$ and set
$\displaystyle T^1{}_2=e^{\textstyle-\int(\Gamma_{12}{}^1-\Gamma_{22}{}^2)dx^2}$.
Then 
$$
\begin{array}{rcl}
T^1{}_{2;1} &=& \partial_{x^1}T^1{}_2+T^1{}_2\Gamma_{11}{}^1\\
\noalign{\medskip}
&=& e^{\textstyle-\int(\Gamma_{12}{}^1-\Gamma_{22}{}^2)dx^2}\left(
-\partial_{x^1}\int(\Gamma_{12}{}^1-\Gamma_{22}{}^2)dx^2 +\Gamma_{11}{}^1\right)=0\,,
\end{array}
$$
thus showing that $T$ is a nilpotent K\"ahler structure.
\end{proof}

\begin{observation}
Let $\mathcal{M}$ be a simply connected affine surface with $\operatorname{Rank}\{\rho_s\}=1$. 
The following conditions are equivalent:
\begin{enumerate}
\item  $\nabla\!\rho_s\!=\!\omega\otimes\rho_s$.
\item $\nabla\!\ker\{\rho_s\!\}\!\subset\!\ker\{\rho_s\!\}$.\hfill
\item $\ker\{\rho_s\!\}\!=\!\operatorname{Span}\{\! X\!\}$ and $\nabla\! X\!=\!\eta\otimes X$.
\end{enumerate}\end{observation}

\begin{proof}
Assume that $\operatorname{Rank}\{\rho_s\}=1$. Choose local coordinates so that the symmetric Ricci tensor has the form
$\rho_s=\rho_{s,22}dx^2\otimes dx^2$. A straightforward calculation shows that any of the 
 conditions of the observation is equivalent to the condition $\Gamma_{11}{}^2=\Gamma_{12}{}^2=0$.
\end{proof}

Consequently, if the $\rho_s$ has rank one and if
$\ker(\rho_s)$ is  parallel, then the affine surface admits a nilpotent K\"ahler structure (see, for example \cite{Op1}).

\begin{remark}\label{EGR-3.5}
\rm
An affine surface $\mathcal{M}$ is projectively flat if and only if both $\rho$ and $\nabla\rho$ are totally symmetric. 
Projective flatness is a specially relevant condition when considering Riemannian extensions, since 
$(T^*M,g_{\nabla,\phi,T})$ as in Theorem~\ref{T1.1} with $\phi=0$ and $T=0$ is locally conformally flat if and only if 
$\mathcal{M}$ is projectively flat (see \cite{EGR-Afifi}).
Let $\mathcal{M}$ be given by the relations in Equation~\eqref{E3.b}. Then $\mathcal{M}$ is projectively flat if and only if 
$\partial_{x^1}\Gamma_{22}{}^2=0$ (equivalently, $\rho$ is symmetric) and $\partial_{x^1,x^1}\Gamma_{22}{}^1=0$
(equivalently,$\nabla\rho=\omega\otimes\rho$ with $\omega(\ker\{\rho\})=0$). 
A straightforward calculation shows that the modified Riemannian extension
$(T^*M,g_{\nabla,\phi,T})$ in Theorem~\ref{T1.1} with $\phi=0$ is anti-self-dual if and only if $(M,\nabla)$ is projectively flat.
\end{remark}

\subsection{K\"ahler structures}\label{S3.3}

\begin{lemma}\label{EGR-L3.6} 
Let $(M,\nabla)$ be an affine surface which is not flat.
\begin{enumerate}
\item If $\mathcal{M}$ admits a K\"ahler structure, then there is a coordinate atlas so
\begin{equation}\label{E3.d}
\Gamma_{11}{}^1=\Gamma_{12}{}^2=-\Gamma_{22}{}^1\,,\quad
\Gamma_{11}{}^2=-\Gamma_{12}{}^1=-\Gamma_{22}{}^2\,.
\end{equation}
\item If Equation~(\ref{E3.d}) holds, then
$$
\rho_s=(\partial_{x^2}\Gamma_{11}{}^2-\partial_{x^1}\Gamma_{11}{}^1)\left(\begin{array}{cc}1&0\\0&1\end{array}\right)\text{ and }
T=\left(\begin{array}{cc}0&-1\\1&0\end{array}\right)\in\mathcal{P}^0(\mathcal{M})\,.
$$
\item If Equation~(\ref{E3.d}) holds and if $\dim\{\mathcal{P}^0(\mathcal{M})\}\ge2$, there exists smooth $\psi$ so 
\begin{equation}\label{E3.e}
\Gamma_{11}{}^1=\frac{1}{2}\,\partial_{x^2}\psi\,\,\text{ and }\,\,\Gamma_{11}{}^2=\frac{1}{2}\,\partial_{x^1}\psi\,.
\end{equation}
\item If Equation~(\ref{E3.d}) and (\ref{E3.e}) hold, then $\rho=\frac{1}{2}(\partial_{x^1}\partial_{x^1}+\partial_{x^2}\partial_{x^2})\psi\,
dx^1\wedge dx^2$ and
\begin{eqnarray*}
&&\mathcal{P}^0(\mathcal{M})=\operatorname{Span}\left\{\!\left(\!\begin{array}{cc}0&-1\\1&0\end{array}\!\right), 
\left(\!\begin{array}{cc}\cos\psi&-\sin\psi\\-\sin\psi&-\cos\psi\end{array}\!\right),
\left(\!\begin{array}{cc}\sin\psi&\cos\psi\\\cos\psi&-\sin\psi\end{array}\!\right)\!\right\}.
\end{eqnarray*}
\end{enumerate}
\end{lemma}


\begin{proof} Suppose $T\in\mathcal{P}^0(\mathcal{M})$ satisfies $T^2=-\operatorname{id}$. By Lemma~\ref{L2.1},
we can choose local coordinates so $T=\partial_{x^2}\otimes dx^1-\partial_{x^1}\otimes dx^2$. Setting $\nabla T=0$ yields the relations:
\begin{eqnarray*}
\nabla_{\partial x^1} T=0:&\left(\begin{array}{cc}\Gamma_{11}{}^2+\Gamma_{12}{}^1&-\Gamma_{11}{}^1+\Gamma_{12}{}^2\\
-\Gamma_{11}{}^1+\Gamma_{12}{}^2&-\Gamma_{11}{}^2-\Gamma_{12}{}^1\end{array}\right)&=\left(\begin{array}{cc}0&0\\0&0\end{array}\right)\\
\nabla_{\partial x^1} T=0:&\left(\begin{array}{cc}\Gamma_{12}{}^2+\Gamma_{22}{}^1&-\Gamma_{12}{}^1+\Gamma_{22}{}^2\\
-\Gamma_{12}{}^1+\Gamma_{22}{}^2&-\Gamma_{12}{}^2-\Gamma_{22}{}^1\end{array}\right)&=\left(\begin{array}{cc}0&0\\0&0\end{array}\right)\,.
\end{eqnarray*}
These relations establish Equation~(\ref{E3.d}). A direct computation establishes Assertion~(2).
Suppose $\dim\{\mathcal{P}(\mathcal{M})\}\ge3$. Choose $S\in\mathcal{P}^0(\mathcal{M})$ to be linearly independent of $T$. Express
$$
S=\left(\begin{array}{cc}S^1{}_1&S^1{}_2\\
S^2{}_1&-S^1{}_1\end{array}\right),\quad
T=\left(\begin{array}{cc}0&-1\\1&0\end{array}\right),\quad
S+\varepsilon T=\left(\begin{array}{cc}S^1{}_1&S^1{}_2-\varepsilon\\S^2{}_1+\varepsilon&-S^1{}_1\end{array}\right)\,.
$$
We have
$\det(S+\varepsilon T)=\varepsilon^2+\varepsilon(S^2{}_1-S^1{}_2)-(S^1{}_1)^2-S^2{}_1S^1{}_2$.
We use the quadratic formula to solve the equation $\det(S+\varepsilon T)=0$ setting:
\begin{eqnarray*}
\varepsilon&=&\textstyle\frac12\left\{(S^1{}_2-S^2{}_1)\pm\sqrt{(S^1{}_2+S^2{}_1)^2
+4(S^1{}_1)^2}\right\}\,.
\end{eqnarray*}
Since $S$ and $T$ are assumed linearly independent, $S+\varepsilon T$ is a non-trivial
nilpotent element. We can then apply Lemma~\ref{L3.2} and  Assertion~(2) to see $\rho_s=0$ and
derive the relations of Equation~(\ref{E3.e}). This proves Assertion~(3);
Assertion~(4) follows by a direct computation.
\end{proof}

\begin{proof}[Proof of Theorem~\ref{T1.5b-EGR}~(1)]
Let $\mathcal{M}$ be an affine surface with $\rho_s\neq 0$ admitting a K\"ahler structure. Take local coordinates as in Lemma~\ref{EGR-L3.6}. Then the relations in Equation~(\ref{E3.d}) show that $\det\{\rho_s\}>0$ and $\rho_s$ recurrent, i.e., $\nabla\rho_s=\omega\otimes\rho_s$ with 
$$
\omega=-(2\Gamma_{11}{}^1-\partial_{x^1}\log\rho_{s,11})\, dx^1-(2\Gamma_{11}{}^2-\partial_{x^2}\log\rho_{s,22})\, dx^2\,.
$$
Conversely, if $\rho_s$ is recurrent and $\det\{\rho_s\}>0$, there exist local coordinates $(x^1,x^2)$ so that $\rho_s=\psi(x^1,x^2)\,(dx^1\otimes dx^1+dx^2\otimes dx^2)$, see for example Theorem 3.2 in \cite{W64}. Now a straightforward calculation using $\nabla\rho_s=\omega\otimes\rho_s$ gives the relations of Equation~\eqref{E3.d} and thus Assertion~(2) in Lemma~\ref{EGR-L3.6} shows that $\mathcal{M}$ is K\"ahler.
\end{proof}

\subsection{Para-K\"ahler structures}\label{S3.4}

\begin{lemma}\label{EGR-L3.7} 
Let $(M,\nabla)$ be an affine surface which is not flat.
\begin{enumerate}
\item If $\mathcal{M}$ admits a para-K\"ahler structure, then there is a coordinate atlas so
\begin{equation}\label{E3.f}
\Gamma_{11}{}^2=0,\ \Gamma_{12}{}^1=0,\ \Gamma_{12}{}^2=0,\ \Gamma_{22}{}^1=0\,.
\end{equation}
\item If Equation~(\ref{E3.f}) holds, then
$$
T=\left(\begin{array}{cc}1&0\\0&-1\end{array}\right)\in\mathcal{P}^0(\mathcal{M})\text{ and }
\rho_s=-\frac12(\partial_{x^2}\Gamma_{11}{}^1+\partial_{x^1}\Gamma_{22}{}^2)\left(\begin{array}{cc}0&1\\1&0
\end{array}\right)\,.
$$
\item If Equation~(\ref{E3.f}) holds and if $\dim\{\mathcal{P}^0(\mathcal{M})\}\ge2$, then there exists a locally
defined smooth function $\theta$ so
\begin{equation}\label{E3.g}
\Gamma_{11}{}^1=\partial_{x^1}\theta\,\,\text{ and }\,\,\Gamma_{22}{}^2=-\partial_{x^2}\theta\,.
\end{equation}
\item If Equations~(\ref{E3.f}) and (\ref{E3.g}) hold, then $\rho=\partial_{x^1}\partial_{x^2}\theta\, dx^1\wedge dx^2$ and
$$
\mathcal{P}^0(\mathcal{M})=\operatorname{Span}\left\{
\left(\begin{array}{cc}1&0\\0&-1\end{array}\right),
e^{-\theta}\left(\begin{array}{cc}0&1\\0&0\end{array}\right),
e^\theta\left(\begin{array}{cc}0&0\\1&0\end{array}\right)\right\}\,.
$$
\end{enumerate}\end{lemma}


\begin{proof} Let $T\in\mathcal{P}^0(\mathcal{M})$ satisfy $T^2=\operatorname{id}$.
We apply Lemma~\ref{L2.1} to see we may choose local coordinates so
$T=\partial_{x^1}\otimes dx^1-\partial_{x^2}\otimes dx^2$. Setting $\nabla T=0$ yields the relations
\begin{eqnarray*}
\nabla_{\partial x^1} T=0:&\left(\begin{array}{cc}0&-2\Gamma_{12}{}^1\\2\Gamma_{11}{}^2&0\end{array}\right)&=\left(\begin{array}{cc}0&0\\0&0\end{array}\right)\\
\nabla_{\partial x^1} T=0:&\left(\begin{array}{cc}0&-2\Gamma_{22}{}^1\\2\Gamma_{12}{}^2&0\end{array}\right)&=\left(\begin{array}{cc}0&0\\0&0\end{array}\right)\,.
\end{eqnarray*}
This yields Equation~(\ref{E3.f}). 
Suppose $\dim\{\mathcal{P}^0(\mathcal{M})\}\ge2$. If $\dim\{\mathcal{P}^0(\mathcal{M})\}=3$,
then $\mathcal{P}^0(\mathcal{M})$ contains a nilpotent element and we may apply Lemma~\ref{L3.2} to conclude
$\rho_s=0$ and Assertion~(2) gives the relations of Equation~(\ref{E3.g}) for suitably chosen $\theta$.
We therefore suppose $\dim\{\mathcal{P}^0(\mathcal{M})\}=2$. Let $\{S,T\}$ be linearly independent elements of $\mathcal{P}^0(\mathcal{M})$. Expand
$$
S=\left(\begin{array}{cc}S^1{}_1&S^1{}_2\\
S^2{}_1&-S^1{}_1\end{array}\right),\quad
T=\left(\begin{array}{cc}1&0\\0&-1\end{array}\right),\quad
ST=\left(\begin{array}{cc}S^1{}_1&-S^1{}_2\\S^2{}_1&S^1{}_1\end{array}\right)\,.
$$
Since $\operatorname{Tr}(ST)=2S^1{}_1$ is constant, we obtain $S^1{}_1$ is constant. Define $\widehat{S}=S-S^1{}_1\,T$. Then $\widehat{S}$ is parallel and $\widehat{S}\neq 0$ since $S$ and $T$ are linearly independent. We then have
$$
\widehat{S}=\left(\begin{array}{cc}0&S^1{}_2\\
S^2{}_1&0\end{array}\right),\quad
T=\left(\begin{array}{cc}1&0\\0&-1\end{array}\right),\quad
\widehat{S}T=\left(\begin{array}{cc}0&-S^1{}_2\\
S^2{}_1&0\end{array}\right)\,.
$$
Since $\widehat{S}\pm \widehat{S}T$ are nilpotent and not both are zero, $\mathcal{P}(\mathcal{M})$ contains a non-trivial
nilpotent element and we can use Lemma~\ref{L3.2} to conclude $\rho_s=0$ and (2) establishes Assertion~(3).
Assertion~(4) follows by a direct computation.
\end{proof}

\begin{proof}[Proof of Theorem~\ref{T1.5b-EGR}~(2)]
Let $\mathcal{M}$ be an affine surface with $\rho_s\neq 0$ admitting a para-K\"ahler structure. 
Take local coordinates as in Lemma~\ref{EGR-L3.7}. Then the relations in Equation~(\ref{E3.f}) 
show that $\det\{\rho_s\}<0$ and $\rho_s$ recurrent, i.e., $\nabla\rho_s=\omega\otimes\rho_s$ with
$$
\omega=-(\Gamma_{11}{}^1-\partial_{x^1}\log\rho_{s,12})\, dx^1-(\Gamma_{22}{}^2-\partial_{x^2}\log\rho_{s,12})\, dx^2\,.
$$
Conversely, if $\rho_s$ is recurrent and $\det\{\rho_s\}<0$, there exist local coordinates 
$(x^1,x^2)$ so that $\rho_s=\psi(x^1,x^2)\,(dx^1\otimes dx^2+dx^2\otimes dx^1)$, see for 
example Theorem 3.2 in \cite{W64}. Now a straightforward calculation using 
$\nabla\rho_s=\omega\otimes\rho_s$ gives the relations of Equation~\eqref{E3.f} and thus 
Assertion~(2) in Lemma~\ref{EGR-L3.7} shows that $\mathcal{M}$ admits a para-K\"ahler structure.
\end{proof}

\section{Type~$\mathcal{A}$ geometry: the proof of Theorem~\ref{T1.9}}\label{S4}

The Ricci tensor of any Type~$\mathcal{A}$ homogeneous model is symmetric. 
Furthermore, the Ricci tensor is recurrent if and only if it is of rank one (see Lemma 2.3 in \cite{BGG18}).
Therefore Theorem~\ref{T1.5b-EGR}~(3) shows that a Type~$\mathcal{A}$ homogeneous surface 
admits a parallel tensor field  if and only if the Ricci tensor is of rank one, in which case it is a nilpotent K\"ahler surface.

The constructions in Theorem \ref{T1.2} and Theorem \ref{T1.3} make an explicit use of the 
nilpotent K\"ahler structure. Therefore, it is important to have concrete expressions.
We begin with a useful algebraic fact that we will use to explicitely determine all
nilpotent K\"ahler structures on Type~$\mathcal{A}$ homogeneous models.

\begin{lemma}\label{L4.1}
Let $\nabla$ be a Type~$\mathcal{A}$ connection on $M=\mathbb{R}^2$ which is not flat and which satisfies $\mathcal{P}^0(\mathcal{M})\ne\{0\}$.
There exists $(a_1,a_2)\in\mathbb{R}^2$ 
and $0\ne\mathfrak{t}\in M_2^0(\mathbb{R})$ so that $\mathcal{P}^0(\mathcal{M})=e^{a_1x^1+a_2x^2}\mathfrak{t}
\cdot\mathbb{R}$.
\end{lemma}

\begin{proof} It is convenient to complexity and set 
$\mathcal{P}_{\mathbb{C}}^0(\mathcal{M}):=\mathcal{P}^0(\mathcal{M})\otimes_{\mathbb{R}}\mathbb{C}$. 
Let $\mathfrak{K}(\mathcal{M})$ be the Lie algebra of affine Killing vector fields.
If $K\in\mathfrak{K}(\mathcal{M})$ and if
$T\in\mathcal{P}_{\mathbb{C}}^0(\mathcal{M})$,
then the Lie derivative $\mathcal{L}_KT$ belongs to $\mathcal{P}_{\mathbb{C}}^0(\mathcal{M})$. Thus
$\mathcal{P}_{\mathbb{C}}^0(\mathcal{M})$ is a finite
dimensional complex $\mathfrak{K}(\mathcal{M})$ module.
If $\nabla$ defines a Type~$\mathcal{A}$ structure on $\mathbb{R}^2$, the Christoffel symbols are constant and
$\partial_{x^1}$ and $\partial_{x^2}$ are affine Killing vector fields.
If $X$ and $Y$ are vector fields, then we have $\mathcal{L}_XY=[X,Y]$ is the Lie bracket.
Thus $\mathcal{L}_{\partial_{x^i}}\partial_{x^j}=0$ and dually $\mathcal{L}_{\partial_{x^i}}dx^j=0$;
if $T=T^i{}_j\partial_{x^i}\otimes dx^j$,
then $\{\mathcal{L}_{\partial_{x^k}}T\}^i{}_j=\partial_{x^k}\{T^i{}_j\}$; the components of $T$ do not interact. The operators 
$\partial_{x^1}$ and $\partial_{x^2}$ commute and act on the finite dimensional vector space $\mathcal{P}_{\mathbb{C}}^0(\mathcal{M})$. 
Consequently, there is a non-trivial joint eigenvector
so $\partial_{x^1}T^i{}_j=a_1T^i{}_j$ and $\partial_{x^2}T^i{}_j=a_2T^i{}_j$; this implies 
$T=e^{a_1x^1+a_2x^2}\mathfrak{t}$ 
for $0\ne \mathfrak{t}\in M_2^0(\mathbb{C})$. Since $\mathcal{M}$ is not flat, the Ricci tensor is nonzero. 
Since the Ricci tensor is symmetric
for a Type~$\mathcal{A}$ geometry, $\rho_s\ne0$. Theorem~\ref{T1.5} then implies 
$\dim\{\mathcal{P}^0_{\mathbb{C}}(\mathcal{M})\}=1$.
Thus the real and imaginary parts of $T$ are linearly dependent and we can assume $T$ is real. 
The desired result now follows.
\end{proof}

Theorem \ref{T1.9} will follow from the following result.

\begin{lemma}
Let $\mathcal{M}=(\mathbb{R}^2,\nabla)$ be a Type~$\mathcal{A}$ structure which is not flat.  Then
$\mathcal{P}^0(\mathcal{M})\ne\{0\}$ if and only if $\mathcal{M}$ is linearly equivalent to a Type~$\mathcal{A}$ structure 
with $\Gamma_{11}{}^2=\Gamma_{12}{}^2=0$. In this setting,
$$
\rho=(-\Gamma_{12}{}^1\Gamma_{12}{}^1+\Gamma_{11}{}^1\Gamma_{22}{}^1+\Gamma_{12}{}^1\Gamma_{22}{}^2)
dx^2\otimes dx^2\,.
$$
Let $a_1:=-\Gamma_{11}{}^1$, let $a_2:=\Gamma_{22}{}^2-\Gamma_{12}{}^1$, and let
$T=e^{a_1x^1+a_2x^2}\partial_{x^1}\otimes dx^2$. Then
$\mathcal{P}^0(\mathcal{M})=T\cdot\mathbb{R}$ is 1-dimensional and nilpotent.
\end{lemma}

\begin{proof}
Let $\nabla$ define a Type~$\mathcal{A}$ structure on $\mathbb{R}^2$ with $\mathcal{P}^0(\mathcal{M})\ne\{0\}$ which is not flat.
We apply Lemma~\ref{L4.1} to choose $(a_1,a_2)$ so that 
$0\ne T=e^{a_1x^1+a_2x^2}\mathfrak{t}\in\mathcal{P}_{\mathbb{C}}^0(\mathcal{M})$ for some 
$0\ne\mathfrak{t}\in M_2^0(\mathbb{C})$.
By Lemma~\ref{L1.4}, the eigenvalues of $T$ are constant. Assume the eigenvalues are nonzero. This implies 
$e^{a_1x^1+a_2x^2}$ is constant and hence $a_1=a_2=0$.
By rescaling $T$, we may assume the eigenvalues are $\pm1$ and hence, after making a complex
 linear change of coordinates,
we may assume $T^1{}_1=1$, $T^2{}_2=-1$, and $T^1{}_2=T^2{}_1=0$. Setting $\nabla T=0$ then yields the relations
$$\Gamma_{12}{}^1=\Gamma_{11}{}^2=\Gamma_{22}{}^1=\Gamma_{12}{}^2=0\,.$$
This forces the Ricci tensor to be zero which is false. Thus no Type~$\mathcal{A}$ geometry which is not flat admits a
K\"ahler or a para-K\"ahler structure.

We may therefore assume the eigenvalues of $T$ are constant and zero. After making a linear change of coordinates,
we can assume $T=e^{a_1x^1+a_2x^2}\partial_{x^1}\otimes dx^2$. We compute $\nabla T=0$ if and only if
$$
\begin{array}{cc}
\Gamma_{11}{}^2=0,&a_1+\Gamma_{11}{}^1-\Gamma_{12}{}^2=0\\
\Gamma_{12}{}^2=0,&a_2+\Gamma_{12}{}^1-\Gamma_{22}{}^2=0.\end{array}
$$
Thus $\mathcal{M}$ admits a non-trivial parallel nilpotent tensor of type $(1,1)$ if and only if $\Gamma_{11}{}^2=\Gamma_{12}{}^2=0$.
We make a direct computation to determine $\rho$. Since the Ricci tensor is symmetric, we use Theorem~\ref{T1.5} to see $\dim\{\mathcal{P}^0(\mathcal{M})\}=1$.
\end{proof}

Results of \cite{BGG18} show that if $\mathcal{M}$ is a Type~$\mathcal{A}$ geometry which
is not flat, then either $\operatorname{dim}\{\mathfrak{K}(\mathcal{M})\}=2$ or
$\operatorname{dim}\{\mathfrak{K}(\mathcal{M})\}=4$.

\begin{corollary} Let $\mathcal{M}=(\mathbb{R}^2,\nabla)$ be a Type~$\mathcal{A}$ structure. The
following assertions are equivalent.
\newline {\rm(1)} $\operatorname{Rank}\{\rho\}=1$.\hfill
{\rm(2)} $\mathcal{P}^0(\mathcal{M})\ne\{0\}$.\hfill
{\rm(3)} $\dim\{\mathcal{P}^0(\mathcal{M})\}=1$.\hfill
{\rm(4)} $\dim\{\mathfrak{K}(\mathcal{M})\}=4$.
\end{corollary}

\begin{proof} Results of \cite{BGG18} (see Lemma 2.3) show that $\rho_s$ has rank 1 if and only if $\mathcal{M}$ is linearly equivalent
to a structure where $\Gamma_{11}{}^2=0$ and $\Gamma_{12}{}^2=0$. The equivalence of Assertion~(1),
Assertion~(2), and Assertion~(3)
then follows from Theorem~\ref{T1.9}. The equivalence of Assertion~(1) and Assertion~(4) follows from Theorem 3.4 of \cite{BGG18}.
\end{proof}
 
 \section{Type~$\mathcal{B}$ geometry}\label{S5}
 
 Let $\mathcal{M}=(\mathbb{R}^+\times\mathbb{R},\nabla)$ where $\Gamma_{ij}{}^k=(x^1)^{-1}C_{ij}{}^k$ and
 $C_{ij}{}^k\in\mathbb{R}$ be a Type~$\mathcal{B}$ surface which is not flat such that $\mathcal{P}^0(\mathcal{M})$ is 
 non-trivial. In Lemma~\ref{L5.1}, we give an algebraic criteria for determining when $\mathcal{P}^0(\mathcal{M})$
 is non-trivial. In Lemmas~\ref{L5.6}--\ref{L5.15}, we use this criteria
 to divide the analysis into 5 different cases and to determine when $\dim\{\mathcal{P}^0(\mathcal{M})\}=1$ or
 $\dim\{\mathcal{P}^0(\mathcal{M})\}=3$. We first prove an analogue of Lemma~\ref{L4.1} in this setting. 
 \begin{lemma}\label{L5.1}
If $\nabla$ is a Type~$\mathcal{B}$ connection on 
 $M=\mathbb{R}^+\times\mathbb{R}$ and if $\mathcal{P}^0(\mathcal{M})\ne\{0\}$,
then there
exists $\alpha\in\mathbb{C}$ and $0\ne\mathfrak{t}\in M_2^0(\mathbb{C})$ so that 
$(x^1)^\alpha\mathfrak{t}\in\mathcal{P}_{\mathbb{C}}^0(\mathcal{M})$.
\end{lemma}

\begin{proof} Let $\nabla$ define a Type~$\mathcal{B}$ structure on $\mathbb{R}^+\times\mathbb{R}$. 
The vector fields
$\partial_{x^2}$ and $X:=x^1\partial_{x^1}+x^2\partial_{x^2}$ are affine Killing vector fields (see \cite{BGG18}). We have:
$$\begin{array}{lll}
\mathcal{L}_X(\partial_{x^i})=[X,\partial_{x^i}]=-\partial_{x^i}\,&\mathcal{L}_X(dx^j)=dx^j,&
\mathcal{L}_X(\partial_{x^i}\otimes dx^j)=0,\\
\mathcal{L}_{\partial_{x^2}}(\partial_{x^i})=0,&\mathcal{L}_{\partial_{x^2}}(dx^j)=0,&
\mathcal{L}_{\partial_{x^2}}(\partial_{x^i}\otimes dx^j)=0.
\end{array}$$

Therefore the components do not interact and we have:
$$
\{\mathcal{L}_XT\}^i{}_j=XT^i{}_j\text{ and }
\{\mathcal{L}_{\partial_{x^2}}T\}^i{}_j=\partial_{x^2}T^i{}_j\,.
$$
Because $\mathcal{P}_{\mathbb{C}}^0(\mathcal{M})$ is a finite dimensional $\partial_{x^2}$ module, we can
find a non-trivial complex eigenvector, i.e. $0\ne T\in\mathcal{P}_{\mathbb{C}}^0(\mathcal{M})$ so
$\partial_{x^2}T^i{}_j=a_2T^i{}_j$. This implies that $T^i{}_j=e^{a_2x^2}t^i{}_j(x^1)$. Applying $X^k$ yields
$$
X^k(T^i{}_j)=e^{a_2x^2}\{a_2^k(x^2)^kt^i{}_j(x^1)+O((x^2)^{k-1})\}\,.
$$
Thus if $a_2\ne0$, the elements $\{T,\mathcal{L}_XT,\dots,\mathcal{L}_XT^k\}$ are linearly independent for any $k$.
This is false since $\dim\{\mathcal{P}_{\mathbb{C}}^0(\mathcal{M})\}\le3$. Therefore,
$T=t^i{}_j(x^1)$. We let $\mathcal{V}\ne\{0\}$ be the subspace of all elements of $\mathcal{P}_{\mathbb{C}}^0(\mathcal{M})$
where $T=T(x^1)$. Choose a non-trivial eigenvector of $\mathcal{L}_X$. Then $x^1\partial_{x^1}T=\alpha T$ implies
$T(x^1)=(x^1)^\alpha\mathfrak{t}$ for some $\mathfrak{t}\in M_2^0(\mathbb{C})$.
\end{proof}

\begin{remark}\label{R5.2}\rm
In the Type~$\mathcal{A}$ setting, the condition $\operatorname{Rank}\{\rho_s\}=1$ implies $\mathcal{P}^0(\mathcal{M})$ is
non-trivial. This fails in the Type~$\mathcal{B}$ setting. Let $\mathcal{M}$ be the Type~$\mathcal{B}$ surface defined by setting 
$C_{22}{}^2=(3+2\sqrt{3})/3$ and $C_{ij}{}^k=1$ otherwise.
We compute that
$$
\rho_s=\frac1{(x^1)^2}\left(\begin{array}{cc}1+\frac{2}{\sqrt{3}}&\frac1{\sqrt{3}} \\ \frac1{\sqrt{3}} & \frac2{\sqrt{3}}-1\end{array}\right)
$$
and consequently $\rho_s$ has rank 1. Assume $\dim\{\mathcal{P}^0(\mathcal{M})\}\ge1$. It follows from Lemma~\ref{L5.1}
that there exists an element
$\mathcal{P}_{\mathbb{C}}^0(\mathcal{M})$ of the form
$T=(x^1)^\alpha(\mathfrak{t}^i{}_j)$ where
$0\ne(\mathfrak{t}^i{}_j)\in M_2^0(\mathbb{C})$. Setting $T^i{}_{j;2}=0$ yields the relations:
$$
(x^1)^{\alpha-1}\left(
\begin{array}{cc}
\mathfrak{t}^2{}_1-\mathfrak{t}^1{}_2 & -2 \mathfrak{t}^1{}_1-\frac{2 }{\sqrt{3}}\mathfrak{t}^1{}_2 \\
 2 \mathfrak{t}^1{}_1+\frac{2 }{\sqrt{3}}\mathfrak{t}^2{}_1 & \mathfrak{t}^1{}_2-\mathfrak{t}^2{}_1\end{array}\right)=
 \left(\begin{array}{cc}0&0\\0&0\end{array}\right)\,.
 $$
 We solve this relation to see $\mathfrak{t}^2{}_1=\mathfrak{t}^1{}_2$ and $\mathfrak{t}^1{}_1=-\frac{1}{\sqrt{3}}\mathfrak{t}^1{}_2$. Substituting these relations and setting
 $T^i{}_{j;1}=0$ then yields:
 $$
(x^1)^{\alpha-1} \left(\begin{array}{cc} -\frac{ \alpha }{\sqrt{3}}\mathfrak{t}^1{}_2 & \left(\alpha +\frac{2}{\sqrt{3}}\right) \mathfrak{t}^1{}_2\\
 \left(\alpha -\frac{2}{\sqrt{3}}\right)\mathfrak{t}^1{}_2  & \frac{ \alpha }{\sqrt{3}}\mathfrak{t}^1{}_2\end{array}\right)
 =\left(\begin{array}{cc} 0 & 0 \\ 0 & 0\end{array}\right)\,.
 $$
This shows $\mathfrak{t}^1{}_2=0$ and hence $T=0$. This shows $\mathcal{P}^0(\mathcal{M})$ is trivial.
The result also follows from Theorem~\ref{T1.5b-EGR} just observing that the symmetric Ricci tensor $\rho_s$ is not recurrent.
\end{remark}

\begin{definition}\rm We follow the discussion of \cite{BGG18} and introduce the following surfaces of Type~$\mathcal{B}$.
\ \begin{enumerate}
\item For $c\in\mathbb{R}$, let $\mathcal{Q}_c$ be the affine manifold of Type~$\mathcal{B}$ defined by
$$
	C_{11}{}^1=0\,,\,\, C_{11}{}^2=c\,,\,\, C_{12}{}^1=1\,,\,\, C_{12}{}^2=0\,,\,\, C_{22}{}^1=0\,,\,\, C_{22}{}^2=1.
$$
Since $\rho=(x^1)^{-2}dx^1\wedge dx^2$,  $\rho_s=0$.
\item For $0\ne c\in\mathbb{R}$, let $\mathcal{P}_{0,c}^\pm$  be the affine manifold of Type~$\mathcal{B}$ defined by
$$\begin{array}{lll}
C_{11}{}^1=\mp c^2+1,& C_{11}{}^2=c,& C_{12}{}^1=0,\\[0.05in]
C_{12}{}^2=\mp c^2,& C_{22}{}^1=\pm 1,& C_{22}{}^2=\pm 2c.
\end{array}$$
Since $\rho=\pm (x^1)^{-2} c\, dx^1\wedge dx^2$, $\rho_s=0$.
\end{enumerate}
\end{definition}

By Theorem~\ref{T1.5}, $\rho_s=0$ if and only if $\dim\{\mathcal{P}^0(\mathcal{M})\}=3$. We give
a complete description of Type~$\mathcal{B}$ manifolds which are not flat where $\rho_s=0$ as follows.

\begin{lemma}\label{L5.4}
\ \begin{enumerate} 
\item If $\mathcal{M}$ is a Type~$\mathcal{B}$ manifold which is not flat but which has $\rho_s=0$, then
$\mathcal{M}$ is linearly equivalent either to $\mathcal{Q}_c$ or to $\mathcal{P}^\pm_{0,c}$.
\item If $\mathcal{M}=\mathcal{Q}_c$ for $c\ne0$,  then
\begin{eqnarray*}
&&\mathcal{P}_{\mathbb{C}}^0(\mathcal{Q}_c)=\operatorname{Span}\left\{
\left(\begin{array}{cc}0&1\\c&0\end{array}\right),\quad
(x^1)^{2\sqrt{c}}\left(\begin{array}{cc}
\sqrt{c}&1
\\-c&-\sqrt{c}
\end{array}\right)\right.,\\
&&\qquad\qquad\qquad\qquad\qquad\qquad\quad\left.(x^1)^{-2\sqrt{c}}\left(\begin{array}{cc}
-\sqrt{c}&1\\-c&\sqrt{c}
\end{array}\right)\right\}\,.
\end{eqnarray*}
\item If $\mathcal{M}=\mathcal{Q}_c$ for $c=0$,  then
\begin{eqnarray*}
&&\mathcal{P}^0(\mathcal{Q}_0)=\operatorname{Span}\left\{
\left(\begin{array}{cc}
0&1\\0&0
\end{array}\right),\quad
\left(\begin{array}{cc}
-\log(x^1)&1-\log(x^1)^2
\\1&-\log(x^1)
\end{array}\right),\right.\\
&&\qquad\qquad\qquad\qquad\qquad\qquad\quad
\left.
\left(\!\begin{array}{cc}
-\log(x^1)&-1-\log(x^1)^2
\\1&-\log(x^1)
\end{array} \!\right)
\!\right\}\,.
\end{eqnarray*}
\item If $\mathcal{M}=\mathcal{P}^\pm_{0,c}$, then
\begin{eqnarray*}
\mathcal{P}^0(\mathcal{P}^\pm_{0,c})&=&\operatorname{Span}\left\{
(x^1)^{-1}\left(\begin{array}{cc}
-c&1\\-c^2&c
\end{array}\right)\right. ,\\
&&\qquad\quad
(x^1)^{-1}\left(\begin{array}{cc}
\pm\frac{1}{2}(x^1 \mp 2c x^2)&x^2\\
\pm c(x^1\mp cx^2)&\mp\frac{1}{2}(x^1 \mp 2c x^2)
\end{array}\right),\\
&&\qquad\quad
\left.
(x^1)^{-1}\left(\begin{array}{cc}
\pm x^2(x^1 \mp c x^2)&(x^2)^2\\
-(x^1\mp cx^2)^2&\mp x^2(x^1 \mp c x^2)
\end{array}\right)\right\}\,.
\end{eqnarray*}
\end{enumerate}
\end{lemma}
\begin{proof}
Assertion~(1) follows from Lemma 4.6 in \cite{BGG18}; the remaining assertions follow from a direct computation.
\end{proof}

\begin{remark}\label{R5.5}\rm
Suppose that $\mathcal{M}$ is a Type~$\mathcal{B}$ surface with $\mathcal{P}^0(\mathcal{M})$ non-trivial.
By Lemma~\ref{L5.1}, there exists $\alpha\in\mathbb{C}$ and $0\ne\mathfrak{t}\in M_2^0(\mathbb{C})$ so that 
$T:=(x^1)^\alpha\mathfrak{t}\in\mathcal{P}_{\mathbb{C}}^0(\mathcal{M})$.   If $\alpha$ is complex, then
the real and imaginary parts of $T$ are linearly dependent and both belong to $\mathcal{P}^0(\mathcal{M})$. 
This implies $\dim\{\mathcal{P}^0(\mathcal{M})\}\ge2$ and hence $\rho_s=0$. Lemma~\ref{L5.4} then yields
$\mathcal{M}=\mathcal{Q}_c$ for $c<0$ and $\alpha$ is purely imaginary.
\end{remark}

In view of Lemma~\ref{L5.4}, we will assume $\rho_s\ne0$ henceforth.
Let $\mathcal{M}$ be a Type~$\mathcal{B}$ geometry with $\mathcal{P}^0(\mathcal{M})$ non-trivial and, since $\rho_s\ne0$,
$\dim\{\mathcal{P}^0(\mathcal{M})\}=1$.
By Lemma~\ref{L5.1}, there exists $\alpha\in\mathbb{C}$ and $0\ne\mathfrak{t}\in M_2^0(\mathbb{C})$ so that 
$(x^1)^\alpha\mathfrak{t}\in\mathcal{P}_{\mathbb{C}}^0(\mathcal{M})$. By Remark~\ref{R5.5}, $\alpha\in\mathbb{R}$ and thus, by 
taking real and imaginary parts,  we may assume that 
 $0\ne\mathfrak{t}\in M_2^0(\mathbb{R})$. Suppose $\alpha=0$. We
 deal with the case $\mathfrak{t}^1{}_2\ne0$ in Lemma~\ref{L5.6}, the case $\mathfrak{t}^1{}_2=0$
and $\mathfrak{t}^2{}_1\ne0$ in Lemma~\ref{L5.8}, and 
the case $\mathfrak{t}^1{}_2=\mathfrak{t}^2{}_1=0$ and $\mathfrak{t}^1{}_1\ne0$ in Lemma~\ref{L5.10}. 
We then turn to the situation where $\alpha\ne0$. Since $\det\{T\}=(x^1)^{2\alpha}\det\{\mathfrak{t}\}$ is constant and since $\alpha\ne0$ is real,
we conclude that $\mathfrak{t}$ is nilpotent. In Lemma~\ref{L5.12}, we assume $\mathfrak{t}^1{}_2\ne0$ and
in Lemma~\ref{L5.15}, we assume $\mathfrak{t}^1{}_2=0$ to complete our analysis.

\begin{lemma}\label{L5.6}
Let $\nabla$ define a Type~$\mathcal{B}$ structure on $\mathbb{R}^+\times\mathbb{R}$ with $\rho_s\ne0$. Suppose that there exists
$0\ne\mathfrak{t}\in\mathcal{P}^0(\mathcal{M})\cap M_2(\mathbb{R})$ with $\mathfrak{t}^1{}_2\ne0$. Rescale $\mathfrak{t}$ to assume that
$\mathfrak{t}^1{}_2=1$.  Then 
$$\begin{array}{ll}
C_{11}{}^1=C_{22}{}^1\, \mathfrak{t}^2{}_1+2(C_{22}{}^2+2C_{22}{}^1 \, \mathfrak{t}^1{}_1)\mathfrak{t}^1{}_1,&
C_{12}{}^1=C_{22}{}^2+2C_{22}{}^1\, \mathfrak{t}^1{}_1,\\[0.05in]
C_{11}{}^2=(C_{22}{}^2+2C_{22}{}^1\, \mathfrak{t}^1{}_1)\mathfrak{t}^2{}_1,& C_{12}{}^2=C_{22}{}^1\,\mathfrak{t}^2{}_1\\[0.05in]
\rho_s=(x^1)^{-2}C_{22}{}^1\left(\begin{array}{cc}\mathfrak{t}^2{}_1&-\mathfrak{t}^1{}_1\\-\mathfrak{t}^1{}_1&-1\end{array}\right),&
\  C_{22}{}^1\ne0,\\[0.05in]
\mathcal{P}^0(\mathcal{M})=\left(\begin{array}{cc}\mathfrak{t}^1{}_1&1\\\mathfrak{t}^2{}_1&-\mathfrak{t}^1{}_1\end{array}\right)
\cdot\mathbb{R}\end{array}$$
\end{lemma}

\begin{proof} 
The equations $\nabla_{\partial_{x^i}}\mathfrak{t}=0$, $i=1,2$ become:
\begin{eqnarray*}
&\left(
\begin{array}{cc}
 {C_{12}{}^1} \mathfrak{t}^2{}_1 -{C_{11}{}^2} & {C_{11}{}^1}-{C_{12}{}^2}-2 {C_{12}{}^1} \mathfrak{t}^1{}_1  \\
 -{C_{11}{}^1} \mathfrak{t}^2{}_1 +{C_{12}{}^2} \mathfrak{t}^2{}_1 +2 {C_{11}{}^2} \mathfrak{t}^1{}_1  & {C_{11}{}^2}-{C_{12}{}^1} \mathfrak{t}^2{}_1  \\
\end{array}
\right)
=\left(\begin{array}{cc}0&0\\0&0\end{array}\right),
\\
&\left(
\begin{array}{cc}
 {C_{22}{}^1} \mathfrak{t}^2{}_1 -{C_{12}{}^2} & {C_{12}{}^1}-{C_{22}{}^2}-2 {C_{22}{}^1} \mathfrak{t}^1{}_1  \\
 -{C_{12}{}^1} \mathfrak{t}^2{}_1 +{C_{22}{}^2} \mathfrak{t}^2{}_1 +2 {C_{12}{}^2} \mathfrak{t}^1{}_1  & {C_{12}{}^2}-{C_{22}{}^1} \mathfrak{t}^2{}_1  \\
\end{array}
\right)=\left(\begin{array}{cc}0&0\\0&0\end{array}\right)\,.
\end{eqnarray*}
These equations yield the relations amongst the $C_{ij}{}^k$; a direct computation then yields $\rho_s$; we obtain $C_{22}{}^1\ne0$
since $\rho_s\ne0$. Furthermore, since $\rho_s\ne0$,  we have
$\dim\{\mathcal{P}^0(\mathcal{M})\}=1$ and the element given spans $\mathcal{P}^0(\mathcal{M})$.
\end{proof}
\begin{remark}
\rm
Let $\mathfrak{t}$ be a nilpotent K\"ahler tensor field as in Lemma \ref{L5.6}. 
Then, in contrast with Remark \ref{R1.11}, the modified Riemannian extension $(T^*M,g_{\nabla,0,\mathfrak{t}})$ is never anti-self-dual. 
Indeed, the affine structures in Lemma~\ref{L5.6} are never projectively flat unless $\rho_s=0$ (see Remark~\ref{EGR-3.5}).
\end{remark}

 \begin{lemma}\label{L5.8}
Let $\nabla$ define a Type~$\mathcal{B}$ structure on $\mathbb{R}^+\times\mathbb{R}$ with $\rho_s\ne0$. Suppose that there exists
$0\ne\mathfrak{t}\in\mathcal{P}^0(\mathcal{M})\cap M_2(\mathbb{R})$ with $\mathfrak{t}^1{}_2=0$ and $\mathfrak{t}^2{}_1\ne0$.
Rescale $\mathfrak{t}$ to assume $\mathfrak{t}^2{}_1=1$. Then
\medbreak\qquad
$C_{11}{}^1=C_{12}{}^2+2C_{11}{}^2\mathfrak{t}^1{}_1,\ C_{12}{}^1=0,\ C_{22}{}^1=0,
\ C_{22}{}^2=-2C_{12}{}^2\mathfrak{t}^1{}_1$,
\medbreak\qquad
$\rho=(x^1)^{-2}C_{12}{}^2\left(\begin{array}{cc}1&-2\mathfrak{t}^1{}_1\\0&0\end{array}\right),\ \  C_{12}{}^2\ne0$,
\ \ 
$\mathcal{P}^0(\mathcal{M})=\left(\begin{array}{cc}\mathfrak{t}^1{}_1&0\\1&-\mathfrak{t}^1{}_1\end{array}\right)\cdot\mathbb{R}$.
\end{lemma}
 
 \begin{proof} Setting $\nabla\mathfrak{t}=0$ yields the relations
 \begin{eqnarray*}
 \left(\begin{array}{cc}C_{12}{}^1&-2C_{12}{}^1\mathfrak{t}^1{}_1\\-C_{11}{}^1+C_{12}{}^2+2C_{11}{}^2\mathfrak{t}^1{}_1&-C_{12}{}^1\end{array}\right)
 =\left(\begin{array}{cc}0&0\\0&0\end{array}\right),\\
 \left(\begin{array}{cc}C_{22}{}^1&-2C_{22}{}^1\mathfrak{t}^1{}_1\\-C_{12}{}^1+C_{22}{}^2+2C_{12}{}^2\mathfrak{t}^1{}_1&-C_{22}{}^1\end{array}\right)
 =\left(\begin{array}{cc}0&0\\0&0\end{array}\right).
 \end{eqnarray*}
 We solve these relations to obtain the relations amongst the $C_{ij}{}^k$. 
  We then compute $\rho$. Since $\rho_s\ne0$, $C_{12}{}^2\ne0$. 
  Furthermore, since $\rho_s\ne0$, 
$\dim\{\mathcal{P}^0(\mathcal{M})\}=1$ and the element given spans $\mathcal{P}^0(\mathcal{M})$.\end{proof}

 \begin{remark}
 \rm
 Modified Riemannian extensions of nilpotent tensor fields in Lemma \ref{L5.8} corresponding to $\mathfrak{t}^1{}_1=0$ are anti-self-dual whenever the deformation tensor field $\phi\equiv 0$.
 In this case Lemma~\ref{L5.8} gives $C_{12}{}^1=0,\ C_{22}{}^1=0,\ C_{22}{}^2=0$, and thus $\mathcal{M}$ is also of Type $\mathcal{A}$ (see Remark \ref{R1.8}). In this case, Remark \ref{R1.11} applies. 
 \end{remark}

 \begin{lemma}\label{L5.10}
Let $\nabla$ define a Type~$\mathcal{B}$ structure on $\mathbb{R}^+\times\mathbb{R}$ with $\rho_s\ne0$. Suppose that
there exists $0\ne\mathfrak{t}\in\mathcal{P}^0(\mathcal{M})\cap M_2(\mathbb{R})$ with 
$\mathfrak{t}^1{}_2=\mathfrak{t}^2{}_1=0$. Rescale $\mathfrak{t}$ to assume $\mathfrak{t}^1{}_1=1$. Then
\medbreak\qquad
$C_{11}{}^2=0,\ C_{12}{}^1=0,\ C_{12}{}^2=0,\ C_{22}{}^1=0$,
\medbreak\qquad
$ \rho=(x^1)^{-2}C_{22}{}^2dx^1\otimes dx^2,\  
 C_{22}{}^2\ne0,\ 
 \mathcal{P}^0(\mathcal{M})=\left(\begin{array}{cc}1&0\\0&-1\end{array}\right)\cdot\mathbb{R}$.
\end{lemma}
 
 \begin{proof} Let $\mathfrak{t}=\left(\begin{array}{cc}1&0\\0&-1\end{array}\right)$.
 Setting $\nabla\mathfrak{t}=0$ yields the relations
 $$
 \left(\begin{array}{cc}0&-2C_{12}{}^1\\2C_{11}{}^2&0\end{array}\right)=
 \left(\begin{array}{cc}0&-2C_{22}{}^1\\2C_{12}{}^2&0\end{array}\right)=
 \left(\begin{array}{cc}0&0\\0&0\end{array}\right)\,.
 $$
 The relations of of Lemma~\ref{L5.10} concerning the $C_{ij}{}^k$ now follow. We determine $\rho$ by a direct
 computation; since $\rho_s\ne0$, $C_{22}{}^2\ne0$. Furthermore, since $\rho_s\ne0$, 
$\dim\{\mathcal{P}^0(\mathcal{M})\}=1$ and the element given spans $\mathcal{P}^0(\mathcal{M})$.
 \end{proof}

 \begin{remark}\rm
Theorem \ref{T1.9} shows that Type $\mathcal{A}$ surfaces  with 
$\dim\{\mathcal{P}^0(\mathcal{M})\}\geq 1$ have $\dim\{\mathcal{P}^0(\mathcal{M})\}=1$ in the non flat case and $\mathcal{P}^0(\mathcal{M})$ is generated by a nilpotent K\"ahler structure.
In opposition, the Type $\mathcal{B}$ geometries in Lemma \ref{L5.6} with $\dim\{\mathcal{P}^0(\mathcal{M})\}=1$ contain K\"ahler, para-K\"ahler and nilpotent K\"ahler examples. On the other hand, the Type $\mathcal{B}$ geometries treated in Lemma \ref{L5.8} and Lemma \ref{L5.10} only admit para-K\"ahler structures.
 \end{remark}
 
 \begin{lemma}\label{L5.12}
 Let $\nabla$ define a Type~$\mathcal{B}$ structure on $\mathbb{R}^+\times\mathbb{R}$ with $\rho_s\ne0$. Suppose that
there exists $0\ne\mathfrak{t}\in M_2(\mathbb{R})$ with $\mathfrak{t}^1{}_2\ne0$ and that there exists $\alpha\ne0$ so that
$(x^1)^\alpha\mathfrak{t}\in\mathcal{P}^0(\mathcal{M})$. Rescale $\mathfrak{t}$
so that $\mathfrak{t}^1{}_2=1$. Then
\medbreak\qquad
$C_{12}{}^1=C_{22}{}^2+2C_{22}{}^1\mathfrak{t}^1{}_1$,\quad
$C_{11}{}^2=\mathfrak{t}^1{}_1(-C_{11}{}^1+\mathfrak{t}^1{}_1(C_{22}{}^2+C_{22}{}^1\mathfrak{t}^1{}_1))$,
\medbreak\qquad
$C_{12}{}^2=-C_{22}{}^1(\mathfrak{t}^1{}_1)^2$,\quad
$\alpha=-C_{11}{}^1 + \mathfrak{t}^1{}_1 (2 C_{22}{}^2 + 3 C_{22}{}^1 \mathfrak{t}^1{}_1)\ne-1$,
\medbreak\qquad
$\rho_s\!=\!-(x^1)^{-2}C_{22}{}^1(1+\alpha)\!\left(\!\!\begin{array}{cc}(\mathfrak{t}^1{}_1)^2&\!\!\mathfrak{t}^1{}_1\\\mathfrak{t}^1{}_1&\!\!1\end{array}\!\!\right)$,\quad $C_{22}{}^1\ne0$,
\medbreak\qquad
$\mathcal{P}^0(\mathcal{M})=\!(x^1)^{\alpha}\!\left(\!\!
\begin{array}{cc}
	\mathfrak{t}^1{}_1&\!\!1\\-(\mathfrak{t}^1{}_1)^2&\!\!-\mathfrak{t}^1{}_1
\end{array}\right)\cdot\mathbb{R}$. 
\end{lemma}

\begin{proof} As noted previously, $\alpha\ne0$ implies $\mathfrak{t}$ is nilpotent. Since we assumed $\mathfrak{t}^1{}_2=1$,
$$
T=(x^1)^\alpha\left(
\begin{array}{cc}
\mathfrak{t}^1{}_1&1\\-(\mathfrak{t}^1{}_1)^2&-\mathfrak{t}^1{}_1
\end{array}\right)\,.
$$
The conditions $\nabla_{\partial_{x^i}} T=0$ $(i=1,2)$ imply
the vanishing of the matrices
\medbreak$\left(
\begin{array}{cc}
-	C_{11}{}^2-(C_{12}{}^1\mathfrak{t}^1{}_1-\alpha)\mathfrak{t}^1{}_1&C_{11}{}^1-C_{12}{}^2+\alpha-2C_{12}{}^1\mathfrak{t}^1{}_1\\
\mathfrak{t}^1{}_1(2C_{11}{}^2+(C_{11}{}^1-C_{12}{}^2-\alpha)\mathfrak{t}^1{}_1)&C_{11}{}^2+(C_{12}{}^1\mathfrak{t}^1{}_1-\alpha)\mathfrak{t}^1{}_1
\end{array}\right)$
\medbreak\noindent and
\medbreak$\left(
\begin{array}{cc}
-C_{12}{}^2-C_{22}{}^1(\mathfrak{t}^1{}_1)^2&C_{12}{}^1-C_{22}{}^2-2C_{22}{}^1\mathfrak{t}^1{}_1\\
\mathfrak{t}^1{}_1(2C_{12}{}^2+(C_{12}{}^1-C_{22}{}^2)\mathfrak{t}^1{}_1)&C_{12}{}^2+C_{22}{}^1(\mathfrak{t}^1{}_1)^2
\end{array}\right)$.
\medbreak\noindent
We solve these relations to obtain the relations amongst the $C_{ij}{}^k$. The expression of $\alpha$ and
$\rho_s$ then follows by a direct computation.  Since $\rho_s\ne0$, we obtain $C_{22}{}^1\ne0$, $\alpha\ne0$,
and $\alpha\ne-1$. Furthermore, since $\rho_s\ne0$, 
$\dim\{\mathcal{P}^0(\mathcal{M})\}=1$ and the element given spans $\mathcal{P}^0(\mathcal{M})$.
\end{proof}

\begin{remark}\rm
Let $T$ be a nilpotent K\"ahler tensor field as in Lemma \ref{L5.12}. The modified Riemannian extension $(T^*M,g_{\nabla,0,T})$ is not anti-self-dual.
\end{remark}

\begin{lemma}\label{L5.15}
Let $\nabla$ define a Type~$\mathcal{B}$ structure on $\mathbb{R}^+\times\mathbb{R}$ with $\rho_s\ne0$.
Suppose that
there exists $0\ne\mathfrak{t}\in M_2(\mathbb{R})$ with $\mathfrak{t}^1{}_2=0$ and that there exists $\alpha\ne0$ so that
$(x^1)^\alpha\mathfrak{t}\in\mathcal{P}^0(\mathcal{M})$. Since $\mathfrak{t}$ is nilpotent, $\mathfrak{t}^1{}_1=0$ and
$\mathfrak{t}^2{}_1\ne0$. Rescale $\mathfrak{t}$
so that $\mathfrak{t}^2{}_1=1$. Then
\medbreak\qquad
$C_{12}{}^1=0,\quad C_{22}{}^1=0,\quad C_{22}{}^2=0,\quad \alpha=C_{11}{}^1-C_{12}{}^2\notin\{ 0,-1\}$,
\medbreak\qquad $\rho=(x^1)^{-2}(1+\alpha)C_{12}{}^2dx^1\otimes dx^1$,
\quad
$\mathcal{P}^0(\mathcal{M})=
(x^1)^{C_{11}{}^1-C_{12}{}^2}\partial_{x^2}\otimes dx^1\cdot\mathbb{R}$.
\end{lemma}

\begin{proof} Setting $\nabla T=0$ yields the vanishing of the matrices
$$
\left(\begin{array}{cc}C_{12}{}^1&0\\-C_{11}{}^1+C_{12}{}^2+\alpha&-C_{12}{}^1\end{array}\right)\text{ and }
\left(\begin{array}{cc}C_{22}{}^1&0\\-C_{12}{}^1+C_{22}{}^2&-C_{22}{}^1\end{array}\right)\,.
$$
The relations amongst the $C_{ij}{}^k$ follows and $\alpha$ is determined. A direct computation
yields the Ricci tensor. Since $\rho=\rho_s\ne0$,
$\dim\{\mathcal{P}^0(\mathcal{M})\}=1$ and the element given spans $\dim\{\mathcal{P}^0(\mathcal{M})\}$.
\end{proof}

\begin{remark}\rm
We note that the structure of Lemma~\ref{L5.15} is also Type~$\mathcal{A}$ (see Remark \ref{R1.8}); this is the only
both Type~$\mathcal{A}$ and Type~$\mathcal{B}$ structure which is not flat with $\mathcal{P}^0(\mathcal{M})\ne\{0\}$ up to linear equivalence.
We also see by inspection that the structures of Lemma~\ref{L5.4}, Lemma~\ref{L5.6}, Lemma~\ref{L5.8}, Lemma~\ref{L5.10}, 
Lemma~\ref{L5.12}, and Lemma~\ref{L5.15} are distinct; there is no intersection amongst these classes.
\end{remark}

\subsection*{Dedication: \rm In memory of the victims of terrorism 
Thursday 17 August 2017 (Barcelona Espana), 
Saturday 12 August 2017 (Charlottesville USA),
etc.}


\begin{thebibliography}{99}

\bibitem{EGR-Afifi}
Z. Afifi,  
Riemann extensions of affine connected spaces, 
\emph{Quart. J. Math., Oxford Ser. (2)} \textbf{5} (1954), 312--320.
	
\bibitem{AMK08} 
T. Arias-Marco and O. Kowalski,
Classification of locally homogeneous affine connections with arbitrary torsion on 2-manifolds,
{\it Monatsh. Math. \bf 153} (2008), 1--18.

\bibitem{Bach}
R. Bach, 	
Zur Weylschen Relativit\"atstheorie und der Weylschen Erweiterung des Kr\"ummung\-stensorbegriffs.
\emph{Math. Z.} \textbf{9} (1921), 110--135.

\bibitem{BVGR}
M. Brozos-V\'{a}zquez and E. Garc\'{i}a-R\'{i}o,
Four-dimensional neutral signature self-dual gradient Ricci solitons, 
\emph{Indiana Univ. Math. J.} \textbf{65} (2016), 1921--1943.


\bibitem{BGG18} M. Brozos-V\'{a}zquez, E. Garc\'{i}a-R\'{i}o, and P. Gilkey, 
Homogeneous affine surfaces: affine Killing vector fields and gradient Ricci solitons, 
{\it J. Math. Soc. Japan} {\bf70} (2018), 25--69. 

\bibitem{BVGRGVR} 
M. Brozos-V\'azquez, E. Garc\'{i}a-R\'{i}o, P. Gilkey, and X. Valle-Regueiro,
Half conformally flat generalized quasi-Einstein manifolds of metric signature $(2,2)$,
 arXiv:1702.06714v1 [math.DG], to appear in 
\emph{Int. J. Math.}

\bibitem{CGGV09} E. Calvi\~no-Louzao, E. Garc\'{\i}a--R\'{\i}o, P. Gilkey, and R. V\'{a}zquez-Lorenzo,
The geometry of modified Riemannian extensions, 
{\it  Proc. R. Soc. Lond. Ser. A Math. Phys. Eng. Sci.} {\bf 465} (2009), 2023-2040.

\bibitem{CGGV} E. Calvino-Louzao, E. Garc\'{i}a-R\'{i}o, I. Gutierrez-Rodriguez, and R. Vazquez-Lorenzo,
Bach-flat isotropic gradient Ricci solitons, 
{\it Pacific J. Math.} {\bf 293} (2018), 75--99.

\bibitem{GRS1}
H.-D. Cao, G. Catino, Q. Chen, C. Mantegazza, and L. Mazzieri, 
Bach-flat gradient steady Ricci solitons, 
\emph{Calc. Var. Partial Differential Equations} \textbf{49} (2014), 125--138.

\bibitem{GRS2}
H.-D. Cao, and Q. Chen,
On Bach-flat gradient shrinking Ricci solitons,
\emph{Duke Math. J.} \textbf{162} (2013), 1149--1169.


\bibitem{chen-wang}
X. Chen, and Y. Wang, 
On four-dimensional Anti-self-dual Gradient Ricci solitons, 
\emph{J. Geom. Anal.} \textbf{25} (2015), 1335--1343. 

\bibitem{C04} 
V. Cort\'es, C. Mayer, T. Mohaupt, and F. Saueressig, 
Special geometry of Euclidean supersymmetry 1. Vector multiplets, 
{\it J. High Energy Physics} 2004, 028.

\bibitem{DGP18}
D. D'Ascanio, P. Gilkey, and P. Pisani,
The geometry of locally symmetric affine surfaces,
 arXiv:1706.04958v1 [math.DG], to appear in \emph{Vietnam J. Math.} (Zeidler memorial volume).


\bibitem{Derdzinski}
A. Derdzinski, 
Connections with skew-symmetric Ricci tensor on surfaces,
{\it Results Math.} {\bf52} (2008), 223--245.


	\bibitem{Fox}
	J. F. Fox,
	Remarks on symplectic sectional curvature,
	\emph{Differential Geom. Appl.} \textbf{50} (2017), 52--70.
	
	\bibitem{GRS}
	I. Gelfand, V. Retakh, and M. Shubin, 
	Fedosov manifolds,
	\emph{Adv. Math.} \textbf{136} (1998), 104--140. 
	
	
	\bibitem{Jelonek}
	W. Jelonek,
	Affine surfaces with parallel shape operators,
	\emph{Ann. Polon. Math.} \textbf{56} (1992),179--186.


\bibitem{KOV}
O. Kowalski, B. Opozda, and Vl\'a\v{z}ek,  
A classification of locally homogeneous affine connections with skew-symmetric Ricci tensor on 2-dimensional manifolds
\emph{Monatsh. Math.} \textbf{130} (2000), 109--125. 

\bibitem{NN57} 
A. Newlander and L. Nirenberg, 
Complex analytic coordinates in almost complex manifolds,
{\it Ann. of Math. \bf 65} (1957), 391--404.

\bibitem{Op04}
B. Opozda, 
A classification of locally homogeneous connections on $2$-dimensional manifolds, 
\emph{Differential Geom. Appl.} \textbf{21} (2004), 173--198.

\bibitem{Op1}
B. Opozda,
A class of projectively flat surfaces,
\emph{Math. Z.} \textbf{219} (1995), 77--92.

\bibitem{W64} 
Y.-C. Wong,   
Two dimensional linear connexions with zero torsion and recurrent curvature, 
\emph{Monatsh. Math.} \textbf{68} (1964), 175--184. 
\end{thebibliography}
\end{document}